\documentclass{article}

\usepackage[verbose=true,letterpaper]{geometry}
\AtBeginDocument{
  \newgeometry{
    textheight=9in,
    textwidth=5.5in,
    top=1in,
    headheight=12pt,
    headsep=25pt,
    footskip=30pt
  }
}
\usepackage[utf8]{inputenc} 
\usepackage[T1]{fontenc}    
\usepackage{hyperref}       
\usepackage{url}            
\usepackage{booktabs}       
\usepackage{amsfonts}       
\usepackage{nicefrac}       
\usepackage{microtype}      
\usepackage{xcolor}         
\usepackage{amsmath}
\usepackage{amssymb}
\usepackage{cleveref}
\usepackage{xcolor}
\usepackage{graphicx}
\usepackage{tikz}
\usetikzlibrary{positioning,shadows,fit,decorations.pathmorphing,matrix}
\usepackage{wrapfig}
\usepackage[noend]{algorithmic}
\usepackage[ruled,noend]{algorithm2e}
\usepackage{etoc}

\definecolor{linkcolor}{RGB}{83,83,182}
\definecolor{citecolor}{RGB}{128,0,128}
\hypersetup{
    colorlinks=true,
    citecolor=citecolor,
    linkcolor=linkcolor,
    urlcolor=linkcolor
}

\newcommand{\fix}{\mathrm{fix}}
\newcommand{\fixj}{\mathrm{fix}(\mathcal{J})}
\newcommand{\op}{\mathrm{op}}

\newcommand{\conv}{\mathrm{conv}}
\newcommand{\dist}{\mathrm{dist}}
\newcommand{\gap}{\mathrm{gap}}
\newcommand{\prox}{\mathrm{prox}}

\newcommand{\J}{\mathcal{J}}

\newcommand{\Xc}{\mathcal{X}}
\newcommand{\Zc}{\mathcal{Z}}
\newcommand{\Ac}{\mathcal{A}}
\newcommand{\Bc}{\mathcal{B}}
\newcommand{\Yc}{\mathcal{Y}}

\newcommand{\Cc}{\mathcal{C}}

\newcommand{\Lc}{\mathcal{L}}
\newcommand{\Jc}{\mathcal{J}}

\newcommand{\BB}{\mathbb{B}}
\newcommand{\RR}{\mathbb{R}}
\newcommand{\NN}{\mathbb{N}}

\newcommand{\partialc}{\partial^c}

\newcommand{\imp}{\mbox{\scriptsize \rm imp}}
\newcommand{\jimp}{J_{\bar x}^{\imp}}
\newcommand{\alg}{\mbox{\scriptsize \rm pb}}
\newcommand{\jalg}{J_{\bar x}^{\alg}}
\newcommand{\jac}{\mathrm{Jac}\,}

\newcommand{\iidsim}{\overset{\text{i.i.d}}{\sim}}

\newtheorem{theorem}{Theorem}
\newtheorem{lemma}{Lemma}
\newtheorem{proposition}{Proposition}
\newtheorem{corollary}{Corollary}

\newtheorem{definition}{Definition}

\newtheorem{assumption}{Assumption}
\newtheorem{remark}{Remark}

\newtheorem{example}{Example}
\newenvironment{proof}[1][]{\noindent {\bf Proof #1:\;}}{\hfill $\Box$}

\usepackage{enumitem}
\newlist{compactitem}{itemize}{3}
\setlist[compactitem]{topsep=0pt,partopsep=0pt,itemsep=0pt,parsep=0pt}
\setlist[compactitem,1]{label=---}
\setlist[compactitem,2]{label=\textbullet}
\setlist[compactitem,3]{label=*}

\title{Automatic differentiation of\\ nonsmooth iterative algorithms}

\author{%
  Jérôme Bolte\footnote{%
  Toulouse School of Economics, %
  University of Toulouse Capitole.}
\and
  Edouard Pauwels\footnote{%
  IRIT, CNRS, %
  Université de Toulouse, ANITI, Toulouse, France.}
\and
  Samuel Vaiter\footnote{%
  CNRS \& Laboratoire J. A. Dieudonné, %
  Université Côte d'Azur, Nice, France.}
}

\newif\ifdedic
\dedictrue

\begin{document}
\etocdepthtag.toc{mtsection}
\etocsettagdepth{mtsection}{subsection}
\etocsettagdepth{mtappendix}{none}

\maketitle

\begin{abstract}
Differentiation along algorithms, i.e., piggyback propagation of derivatives, is now routinely used to differentiate iterative solvers in differentiable programming. Asymptotics is well understood for many smooth problems but the nondifferentiable case is hardly considered. Is there a limiting object for nonsmooth piggyback automatic differentiation (AD)? Does it have any variational meaning and can it be used effectively in machine learning? Is there a connection with classical derivative? All these questions are addressed under appropriate nonexpansivity conditions in the framework of conservative derivatives which has proved useful in understanding nonsmooth AD. For nonsmooth piggyback iterations, we characterize the attractor set of nonsmooth piggyback iterations as a set-valued fixed point which remains in the conservative framework. This has various consequences and in particular almost everywhere convergence of classical derivatives. Our results are illustrated on parametric convex optimization problems with forward-backward, Douglas-Rachford and Alternating Direction of Multiplier algorithms as well as the Heavy-Ball method.
\end{abstract}

\ifdedic
\begin{quote}
    \emph{%
    Dedicated to the memory of Andreas Griewank -- a pioneer in automatic differentiation and optimization -- who passed away on September 2021.%
    }
\end{quote}
\fi

\section{Introduction}
\begin{wrapfigure}{r}{0.3\textwidth}
  \vspace{-20pt}
  \centering
  \begin{tikzpicture}
    \draw (0,0)  node(f)   {$x_{k}(\theta)$};
		\draw (3.1,0)  node(df)  {$J_{x_k}(\theta)$};
		\draw (0,-1.7) node(fun) {$\bar x(\theta)$};
		\draw (3.1,-1.7) node(dfun) {$\jalg(\theta)$};

		\draw[->] (f) -- (df) node[above,midway]{\footnotesize nonsmooth} node[below,midway]{\footnotesize autodiff};
    \draw[->] (f) -- (fun) node[rotate=90,below,midway] {\tiny $k \to +\infty$};
    \draw[magenta,thick,densely dotted,->,decorate,decoration={snake,amplitude=.4mm,segment length=2mm,post length=1mm,pre length=2mm}] (fun) -- (dfun) node[below,midway]{\footnotesize \color{magenta} derivative?};
		\draw[magenta,thick,densely dotted,->,decorate,decoration={snake,amplitude=.4mm,segment length=2mm,post length=1mm,pre length=2mm}] (df) -- (dfun) node[rotate=90,below,midway] {\footnotesize \color{magenta} limit?}  ;
  \end{tikzpicture}
	\caption{We study existence and meaning of $\jalg$ as a derivative of $\bar{x}$, compatible with automatic differentiation of the iterates $(x_k(\theta))_{k\in \NN}$.}
  \label{fig:ad}
\end{wrapfigure}
\paragraph{Differentiable programming.}
We consider a Lipschitz function $F \colon \RR^p \times \RR^m \mapsto \RR^p$, representing an iterative algorithm, parameterized by $\theta \in \RR^m$, with Lipschitz initialization $x_0\colon \theta \mapsto x_0(\theta)$ and
\begin{equation}\label{eq:iterative}
    x_{k+1} (\theta) = F(x_k(\theta), \theta )
    = F_\theta(x_k(\theta)) ,
\end{equation}
where $F_\theta := F(\cdot, \theta)$, under the assumption that $x_k(\theta)$ converges to the unique fixed point of $F_\theta$: $\bar{x}(\theta) = \fix(F_{\theta})$.
Such recursion represent for instance algorithms to solve an optimization problem $\min_{x} h(x)$ (e.g. empirical risk minimization), such as gradient descent: $F(x,\theta) = x - \theta \nabla h(x)$.
But~\eqref{eq:iterative} could also be a fixed-point equation such as a deep equilibrium network~\cite{bai2019deep}.

In the last years, a paradigm shift occurred: such algorithms are now implemented in algorithmic differentiation (AD)-friendly frameworks such as Tensorflow~\cite{tensorflow2015-whitepaper}, PyTorch~\cite{paszke2019pytorch} or JAX~\cite{jax2018github} to name a few.
Assuming that $F$ is differentiable, it is possible to compute iteratively the derivatives of $x_{k+1}$ with respect to $\theta$ using the differential calculus chain rule resulting in so called ``piggyback'' recursion:
\begin{equation}\label{eq:ad-smooth}
  \frac{\partial}{\partial \theta} x_{k+1} (\theta)
  =
  \partial_{1} F(x_k(\theta), \theta )  \cdot \frac{\partial}{\partial \theta} x_{k} (\theta)
  + \partial_{2} F(x_k(\theta), \theta ) ,
\end{equation}
where $\frac{\partial}{\partial \theta} x_k$ is the Jacobian of $x_k$ with respect to $\theta$.
In practice, automatic differentiation frameworks do not compute the full Jacobian, but compute either vector-Jacobian products in reverse-mode (or backpropagation)~\cite{rumelhart1986learningrepresentations} or Jacobian-vector products in forward mode~\cite{wengert1964simpleautomatic}.
We rather consider the full Jacobian, and therefore, our findings \emph{apply to both} modes.
We focus on two issues arising with nonsmooth recursions, illustrated in \Cref{fig:ad}. \textit{(i)} what can be said about the chain rule~\eqref{eq:ad-smooth} and its asymptotics when the function $F$ is not smooth (for example a projected gradient step)? \textit{(ii)} how to interpret its asymptotics as a notion of derivative for $\bar{x}$, the fixed point of $F_{\theta}$? We propose a \emph{joint} answer to both questions,  providing a solid theoretical ground to the idea of algorithmic differentiation of numerical solvers involving nonsmooth components in a differentiable programming context.

\paragraph{Related works.}
Algorithmic use of the chain rule~\eqref{eq:ad-smooth} to differentiate programs takes its root in~\cite{wengert1964simpleautomatic}, where forward differentiation was first proposed, and later in reverse mode~\cite{linnainmaa1970representationcumulative}.
Along with the practical development of automatic differentiation, the question on how to prove the convergence of the iterative sequence~\eqref{eq:ad-smooth} was investigated, notably in the optimization community as reviewed in~\cite{griewank2003piggyback}.
This is an important paper containing several ideas in differentiable programming rediscovered/reused later: implicit differentiation \cite{pedregosa2016hyperparameter,rajeswaran2019meta} and its stability~\cite{blondel2022efficient}, adjoint fixed point iteration~\cite{bai2019deep} that is a key aspect of the deep equilibrium network and linear convergence of~\eqref{eq:ad-smooth} as discussed below.
Notably, the linear convergence of the Jacobians was investigated in \cite{gilbert1992automatic,griewank1993derivative} for the forward mode and in~\cite[Theorem 2.3]{christianson1994reverse} for the reverse mode.
This was more recently investigated -- for \emph{$C^{2}$ functions} -- in the imaging community for primal-dual algorithms~\cite{chambolle2021learning,bogensperger2022convergence} and in the machine learning community for gradient descent~\cite{mehmood2020automatic,lorraine2020optimizing} and the Heavy-ball~\cite{mehmood2020automatic} method.
Note that in the specific context where $F$ solves a $\min$-$\min$ problem, the authors in~\cite{ablin2020super} proved the linear convergence of the Jacobians.
The use of automatic differentiaton for nonsmooth functions was justified by the development of the notion of \emph{conservative Jacobians}~\cite{bolte2020conservative,bolte2020mathematical} with a nonsmooth version of the chain rule for compositional models.
The correctness of automatic differentiation was also investigated in~\cite{lee2020oncorrectness} for a large class of functions that are piecewise analytic, and also in~\cite{kakade2018provably} where a qualification condition is used to compute a Clarke Jacobian.
Along with automatic differentiation, a natural way to differentiate a model such as~\eqref{eq:iterative} is by implicit differentiation, recently applied in several works~\cite{bai2019deep,agrawal2019differentiable,elghaoui2021implicit}.
To study these models with nonsmooth functions, an implicit function theorem~\cite{bolte2021nonsmooth} was proved for path-differentiable functions.

\begin{figure}[t]
  \centering
  \includegraphics[width=\textwidth]{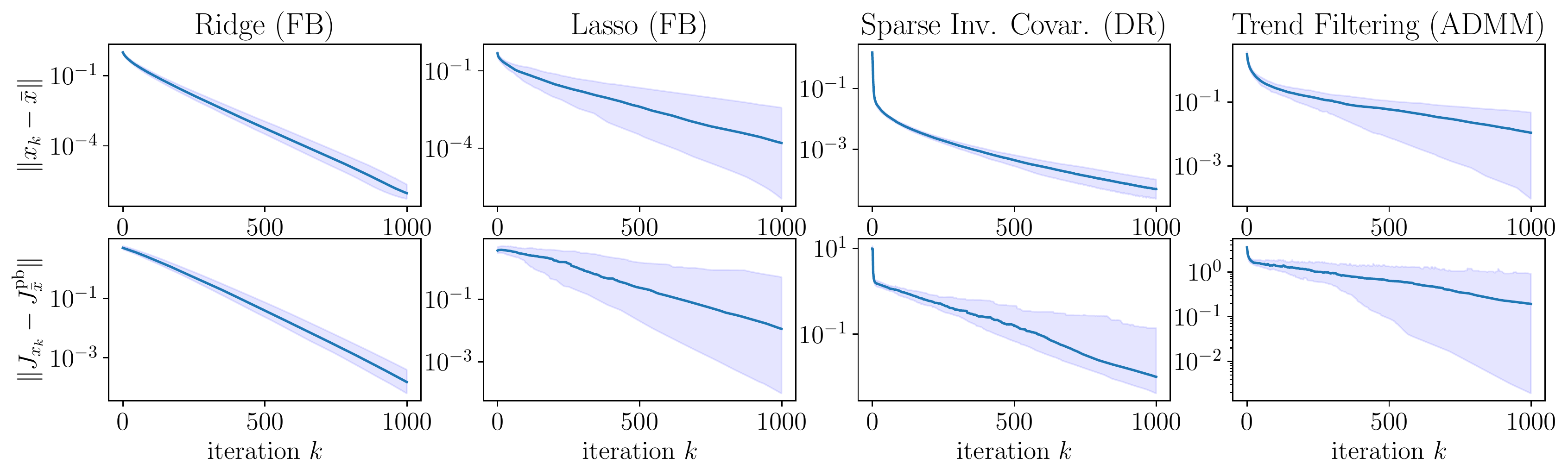}
	\caption{Illustration of the linear convergence of proximal splitting methods. \emph{(First line)} Distance of the iterates to the fixed point. \emph{(Second line)} Distance of the piggyback Jacobians to the Jacobian of the fixed point. The acronyms are FB for Forward-Backward, DR for Douglas-Rachford and ADMM for Alternating Direction Method of Multipliers. In all cases, despite nonsmoothness, piggyback Jacobians converge, illustrating Corollary \ref{cor:classicalJacobians}. Blue lines represent the median of 100 repetitions with random data, and the blue shaded area represents the first and last deciles.}
  \label{fig:linearCvgt}
\end{figure}

\paragraph{Contributions:}{$ $} Under suitable nonexpansivity assumptions, our contributions are as follows.\\
\textbullet\ We address both questions illustrated in Figure \ref{fig:ad} for nonsmooth recursions: set-valued nonsmooth extensions of the piggyback recursion \eqref{eq:ad-smooth} have a well defined limit, described as the fixed point of subset map (\Cref{th:convergenceRecursion}), it is conservative for the fixed point map $\bar{x}$. This is a nonsmooth ``infinite'' chain rule for AD (\Cref{thm:jalgconvervative}).\\
\textbullet\ For almost all $\theta$, despite nonsmoothness, recursion \eqref{eq:ad-smooth} is well defined using the classical Jacobian and converges to the classical Jacobian of the fixed point $\bar{x}$ (\Cref{cor:classicalJacobians}). This has implications for both forward and reverse modes of AD.\\
\textbullet\ For a large class of functions (Lipschitz-gradient selection), it is possible to give a quantitative rate estimate (\Cref{cor:lipconv}), namely to prove linear convergence of the derivatives.\\
\textbullet\ We show that these results can be applied to proximal splitting algorithms in nonsmooth convex optimization. We include forward--backward~(\Cref{prop:fb}), as well Douglas--Rachford (\Cref{prop:DR}) and ADMM, a numerical illustration of the convergence of derivatives is given in \Cref{fig:linearCvgt}.\\
\textbullet\ We also illustrate that, contrarily to the smooth case, nonsmooth piggy back derivatives of momentum methods such as the Heavy-ball algorithm, may diverge even if the iterates converge linearly (\Cref{prop:hb-fail}).

\paragraph{Notations.}
A function $f: \RR^{p} \to \RR^{m}$ is locally Lipschtiz if, for each $x \in \RR^{n}$, there exists a neighborhood of $x$ on which $f$ is Lipschitz.
Denoting by $R \subseteq \RR^{p}$, the full measure set where $f$ is differentiable, the Clarke Jacobian~\cite{clarke1983optimization} at $x \in \RR^{p}$ is defined as
\begin{equation}
  \jac^c f(x) = \conv
  \left\{
					M \in \RR^{m \times p}, \exists (y_{k})_{k\geq 0} \text{ s.t. } \lim_{k \to \infty} y_{k} = x, y_{k} \in R, \lim_{k \to \infty} \frac{\partial f}{\partial y} (y_{k}) = M
  \right\} .
	\label{eq:clarkeJacobian}
\end{equation}
The Clarke subdifferential, $\partial^c f$ is defined similarly.
Given two matrices $A, B$ with compatible dimension, $[A,B]$ is their concatenation.
For a set $\Xc$, $\conv \Xc$ is its convex hull (the smallest convex set containing $\Xc$). The symbol $\mathbb{B}$ denotes a unit ball, the corresponding norm should be clear from the context.

\section{Nonsmooth piggyback differentiation}
We first show how the use of the notion of \emph{conservative Jacobians} allow us to justify rigorously the existence of a nonsmooth equivalent of piggyback iterations in~\eqref{eq:ad-smooth} that is compatible with AD.

\paragraph{Conservative Jacobians.}
Conservative Jacobians were introduced in~\cite{bolte2020conservative} as a generalization of derivatives to study automatic differentiation of nonsmooth functions.
Given a locally Lipschitz continuous function $f: \RR^{p} \to \RR^{m}$, we say that the set-valued $J: \RR^{p} \rightrightarrows \RR^{p \times m}$ is a \emph{conservative Jacobian} for the \emph{path differentiable} $f$ if $J$ has a closed graph, is locally bounded and nowhere empty with
\begin{equation}
  \label{eq:conservative_def}
  \frac{d}{dt} f(\gamma(t))
  =
  J(\gamma(t))\dot \gamma(t) \quad \text{a.e.}
\end{equation}
for any $\gamma \colon [0,1] \to \RR^p$ absolutely continuous with respect to the Lebesgue measure. Conservative gradients are defined similarly.
We refer to~\cite{bolte2020conservative} for extensive examples and properties of this class of function.
Let us mention that the classes convex functions, definable functions, or semialgebraic functions are contained in the set of path differentiable functions.
Given $D_f \colon \RR^p \rightrightarrows \RR^p$, a conservative gradient for $f \colon \RR^p \to \RR$, we have:

\noindent\textbullet\ (\textbf{Clarke subgradient}), for all $x \in \RR^p$, $\partialc f(x) \subset \conv(D_f(x))$.\\
\textbullet\ (\textbf{Gradient almost everywhere}) $D_{f}(x) = \{ \nabla f(x) \}$ for almost all $x \in \RR^p$.\\
\textbullet\ (\textbf{Calculus}) differential calculus rules preserve conservativity, \textit{e.g.} sum and compositions of conservative Jacobians are conservative Jacobians.

An important point is that $D_{f}$ can be used as a first order optimization oracle for methods of gradient type, while preserving usual convergence guaranties~\cite{bolte2020mathematical}.

\paragraph{Piggyback differentiation of recursive algorithms.} The following is standing throughout the text

\begin{assumption}[The conservative Jacobian of the iteration mapping is a contraction]{\rm
    $F$ is locally Lipschitz, path differentiable, jointly in $(x,\theta)$, and $J_{F}$ is a conservative Jacobian for $F$.
  There exists $0 \leq \rho < 1$, such that for any $(x, \theta) \in \RR^{p} \times \RR^{m}$ and any pair $[A,B] \in J_{F}(x,\theta)$, with $A \in \RR^{p \times p}$ and $B \in \RR^{p \times m}$, the operator norm of $A$ is at most $\rho$.
	\label{ass:contraction}}
\end{assumption}
Under Assumption \ref{ass:contraction}, $F_{\theta}$ is a strict contraction so that $(x_k(\theta))_{k \in \NN}$ converges linearly to $\bar{x}(\theta) = \fix(F_{\theta})$ the unique fixed point of the iteration mapping $F_{\theta}$. More precisely, for all $k \in \NN$, we have
\begin{equation*}
    \|x_k - \bar{x}(\theta)\| \leq \rho^k \frac{\|x_0 - F_{\theta}(x_0)\|}{1 - \rho}.
\end{equation*}
Furthermore, for every $k \in \NN$, let us define the following set-valued piggyback recursion:
\begin{align}\label{eq:autodiff}\tag{PB}
  J_{x_{k+1}}(\theta) &= \left\{A J + B,\,  [A, B] \in J_{F}(x_k(\theta),\theta),\,  J \in J_{x_{k}}(\theta) \right\}.
\end{align}
We will show in \Cref{sec:infinite} that~\eqref{eq:autodiff} plays the same role as~\eqref{eq:ad-smooth} in the nonsmooth setting.
Note that one can recursively evaluates a sequence $J_k \in J_{x_k}$, $ k \in \NN$ as follows
\begin{align}
				J_{k+1} &= A_k J_k + B_k \quad \text{where} \quad [A_k, B_k] \in J_{F}(x_k(\theta),\theta),
  \label{eq:autodiffSection}
\end{align}
which corresponds to the operations actually implemented in nonsmooth AD frameworks.

\begin{remark}[Local contractions]{\rm
				Assumption \ref{ass:contraction} may be relaxed locally as follows: for all $\theta$, the fixed point set $\fix(F_\theta)$
				of the iteration mapping $F_\theta$ is a singleton $\bar{x}_\theta$ such that $x_k(\theta) \to \bar{x}(\theta)$ as $k \to \infty$, and the operator norm condition on $J_F$ in Assumption \ref{ass:contraction}  holds at the point $(\bar{x}(\theta),\theta)$. By graph closedness of $J_F$, in a  neighborhood of $(\bar{x}(\theta), \theta)$, $F_\theta$ is a strict contraction and  the operator norm condition on $J_F$ holds, possibly with a larger contraction factor $\rho$. After finitely many steps, the iterates $(x_k)_{k \in \NN}$ remain in this neighborhood and all our convergence results hold, due to their asymptotic nature.
				\label{rem:local}}
\end{remark}

\section{Asymptotics of nonsmooth piggyback differentiation}\label{sec:infinite}
\subsection{Fixed point of affine iterations}
\paragraph{Gap and Haussdorf distance.}
Being given two nonempty compact subsets $\Xc,\Yc$ of $\RR^p$,  set
\begin{equation*}
  \gap(\Xc,\Yc)=\max_{x\in \Xc} d(x, \Yc)
  \quad \text{where} \quad
  d(x, \Yc) = \min_{y\in \Yc} \|x-y\| ,
\end{equation*}
and define the Hausdorff distance between $\Xc$ and $\Yc$ by
\begin{equation*}
  \dist(\Xc,\Yc)=\max (\gap(\Xc,\Yc),\gap(\Yc,\Xc)).
\end{equation*}
Note that $\gap(\Xc,\Yc) = 0$ if, and only if, $\Xc \subseteq \Yc$, whereas $\dist(\Xc,\Yc) = 0$ if, and only if, $\Xc = \Yc$.
Moreover, $\Xc \subseteq \Yc + \gap(\Xc, \Yc) \mathbb{B}$ where $\mathbb{B}$ is the unit ball.
It means that $\gap(\Xc,\Yc)$ ``measures'' the default of inclusion of $\Xc$ in $\Yc$, see~\cite[Chapter 4]{rockafellar2009variational} for more details.

\paragraph{Affine iterations by packets of matrices.}
Let $\Jc \subset \RR^{p\times (p + m)}$ be a compact subset of matrices such that any matrix of the form $[A,B] \in \Jc$ with $A \in \RR^{p \times p}$ is such that $A$ has operator norm at most $\rho < 1$. We let $\Jc$ act naturally on the matrices of size $p \times m$ as follows $\Jc \colon \RR^{p\times m} \rightrightarrows \RR^{p\times m}$ the function from $\RR^{p \times m}$ to subsets of $\RR^{p \times m}$ which is defined for each $X \in \RR^{p \times n}$ as follows: $\Jc(X) = \{AX + B,\, [A,B] \in \Jc\}$. This  defines a set-valued map through, for any $\Xc \subset \RR^{p\times m}$,
\begin{align}\label{eq:actionJMatrix}
  \Jc(\Xc) = \{A X + B,\, [A,B] \in J,\, X \in \Xc\}.
\end{align}
On the model of recursions of the form~\eqref{eq:autodiff}, we consider sequences $(\Xc_k)_{k \in \NN}$ of subsets of $\RR^{p\times m}$ satisfying the recursion
\begin{equation}\label{eq:aff_iter}
  \Xc_{k+1} = \Jc(\Xc_k).
\end{equation}
We have the following instance of the Banach--Picard theorem (proved in \Cref{app:prop-aff-iter}).
\begin{theorem}[Set-valued affine contractions]\label{th:convergenceRecursion}
  Let $\Jc \subset \RR^{p\times (p+m)}$ be a compact subset of matrices as above with $\rho <1$. Then there is a unique nonempty compact set $\fixj \subset \RR^{p \times m}$ satisfying $\fixj = \Jc(\fixj)$, where the action of $\Jc$ is given in \eqref{eq:actionJMatrix}.

  Let $(\Xc_k)_{k \in \NN}$ be a sequence of  compact subsets of $\RR^{p \times m}$, such that $\Xc_0\neq\emptyset$, and satisfying the recursion~\eqref{eq:aff_iter}.
  We have for all $k \in \NN$
  \begin{align*}
    \dist(\Xc_k, \fixj)& \leq \rho^{k} \frac{\dist(\Xc_0, \Jc(\Xc_0))}{1 - \rho},
  \end{align*}
  where $\dist$ is the Hausdorff distance related to the Euclidean norm on $p \times m$ matrices.
\end{theorem}

\subsection{An infinite chain rule and its consequences}
Define the following set-valued map based on the $\fix$ operator from \Cref{th:convergenceRecursion},
\begin{align*}
    \jalg \colon \theta \rightrightarrows \fix\left[
    J_F(\bar{x}(\theta), \theta)\right].
\end{align*}
where $\bar{x}(\theta)$ is the unique fixed point of the algorithmic recursion.
Note that since $\bar x(\theta)=\fix (F_\theta)$, we also have equivalently that $\jalg$ is the fixed-point of the Jacobian of the fixed-point of $F_{\theta}$:
\begin{align*}
    \jalg \colon \theta \rightrightarrows \fix\left[
    J_F(\fix(F_\theta), \theta)\right] .
\end{align*}
We have the following (proved in \Cref{app:cons-jac}).
\begin{theorem}[A conservative mapping for the fixed point map]\label{thm:jalgconvervative}
	Under \Cref{ass:contraction},
  $\jalg$ is well-defined, and is a conservative Jacobian for the fixed point map $\bar{x}$.
\end{theorem}
Combining with \Cref{th:convergenceRecursion} ensures the convergence of the set-valued piggyback iterations \eqref{eq:autodiff}.
\begin{corollary}[Convergence of the piggyback derivatives]\label{cor:unrollingConvergence}
	Under \Cref{ass:contraction},
  for all $\theta$, the recursion~\eqref{eq:autodiff} satisfies
  \begin{equation}\label{eq:unroll-cvg-fix}
    \lim_{k \to \infty}  \gap(J_{x_k}(\theta), \jalg(\theta))= 0.
  \end{equation}
\end{corollary}
  Unrolling the expression of $J_{x_{k}}$, we can rewrite~\eqref{eq:unroll-cvg-fix} as a set-valued product such that
  \begin{equation*}
    \lim_{K \to +\infty} \gap \left(
      \prod_{k=0}^{K} J_{F}(x_k(\theta), \theta)
      ,
      \jalg(\theta)
    \right) = 0 .
  \end{equation*}
  In plain words, this a limit-derivative exchange result: \emph{Asymptotically, the gap between the automatic differentiation of $x_{k}$ and the derivative of the limit is zero.}
	This implies in particular that the recursion \eqref{eq:autodiffSection} produces bounded sequences and all its accumulation points are in $\jalg$.
	Using the fact that conservative Jacobians equal classical Jacobians almost everywhere \cite{bolte2020conservative}, this implies convergence of derivatives in a classical sense.
\begin{corollary}[Convergence of the classical piggyback derivatives]
				Under \Cref{ass:contraction},
				for almost all $\theta$, the classical Jacobian
				$\frac{\partial}{\partial \theta} x_k(\theta)$, is well defined for all $k$ and converges towards the classical Jacobian of $\bar x$:
  \[ \lim_{k \to \infty}
     \frac{\partial}{\partial \theta} x_k(\theta) = \frac{\partial}{\partial \theta} \bar x(\theta).
  \]
	\label{cor:classicalJacobians}
\end{corollary}

\begin{remark}[Connection to implicit differentiation]{\rm
  The authors in~\cite{bolte2021nonsmooth} proved a qualification-free version of the implicit function theorem.
  Assuming that for every $[A,B] \in J(\bar x(\theta), \theta)$, the matrix $I - A$ is invertible, we have that
  \begin{align}\label{eq:implicitDiff}
    \jimp\colon \theta
    \rightrightarrows
    \left\{(I - A)^{-1} B, [A,B] \in J_F(\bar{x}(\theta), \theta) \right\}
  \end{align}
	is a conservative Jacobian for $\bar x$. Under \Cref{ass:contraction}, one has $\jimp(\theta) \subset \jalg(\theta)$ for all $\theta$.
  Unfortunately, as soon as $F$ is not differentiable, the inclusion may be strict, see details in \Cref{app:impl-diff}.
  \label{rem:implicitDiff}}
\end{remark}
\subsection{Consequence for algorithmic differentiation}
Given $k \in \NN$, $\dot{\theta} \in \RR^m$, $\bar{w}_k \in \RR^p$, the following algorithms allow us to compute $\dot{x}_k = J_k \dot{\theta}$ using the forward mode of automatic differentation (Jacobian Vector Products, JVP), and $\bar{\theta}_k^T = \bar{w}_k^T J_k$ using the backward mode of automatic differentiation (Vector Jacobian Products, VJP).

\begin{algorithm}[H]
  \caption{Algorithmic differentiation of recursion \eqref{eq:iterative}, forward and reverse modes}
	\label{alg:autodiff0}
	\textbf{Input:} $k \in \NN$, $\theta \in \RR^m$, $\dot{\theta} \in \RR^m$, $\bar{w}_k \in \RR^p$, initialization function $x_0(\theta)$, recursion function $F(x,\theta)$, conservative Jacobians $J_F(x,\theta)$ and $J_{x_0}(\theta)$. Initialize:
    $x_0 = x_0(\theta) \in \RR^p$.
   \begin{minipage}[t]{0.35\linewidth}
   \begin{algorithmic}
    \STATE \textbf{Forward mode (JVP):}
    \STATE
     $\dot{x}_0 = J\dot{\theta}$, $J \in J_{x_0}(\theta)$.
    \FOR{$i= 1, \ldots, k$}
        \STATE $x_{i} = F(x_{i-1}, \theta)$
				\STATE $\dot{x}_i = A_{i-1} \dot{x}_{i-1} + B_{i-1}\dot{\theta}$
				\STATE $[A_{i-1}, B_{i-1}] \in J_{F}(x_{i-1},\theta)$
    \ENDFOR
    \STATE \textbf{Return:} $\dot{x}_k$
  \end{algorithmic}
  \end{minipage}%
  \begin{minipage}[t]{0.55\linewidth}
     \begin{algorithmic}
     \STATE \textbf{Reverse mode (VJP):}
     $\bar{\theta}_k = 0$.
    \FOR{$i= 1, \ldots, k$}
        \STATE $x_{i} = F(x_{i-1}, \theta)$
    \ENDFOR
    \FOR{$i= k, \ldots, 1$}
        \STATE $\bar{\theta}_k = \bar{\theta}_k +  B_{i-1}^T \bar{w}_i$
        \quad $\bar{w}_{i-1} = A_{i-1}^T\bar{w}_i$
         \STATE $[A_{i-1}, B_{i-1}] \in J_{F}(x_{i-1},\theta)$
    \ENDFOR
    \STATE $\bar{\theta}_k = \bar{\theta}_k +  J^T \bar{w}_0$, $J \in J_{x_0}(\theta)$
    \STATE \textbf{Return:}
$\bar{\theta}_k$
  \end{algorithmic}
  \end{minipage}
\end{algorithm}
The following result is a consequence of \Cref{cor:classicalJacobians} combined with algorithmic differentiation arguments, its proof is given in Appendix \ref{app:impl-diff}.
\begin{proposition}[Convergence of VJP and JVP]
    Let $k \in \NN$, $\dot{\theta} \in \RR^m$, $\bar{w}_k \in \RR^p$, $x_k \in \RR^p$, $\dot{x}_k \in \RR^p$, $\bar{\theta}_k^T \in \RR^m$ be as in Agorithm \ref{alg:autodiff0} under Assumption \ref{ass:contraction}. Then for almost all $\theta \in \RR^m$, $\dot{x}_k \to \frac{\partial \bar{x}}{\partial \theta} \dot{\theta}$.

		Assume furthermore that, as $k \to \infty$, $\bar{w}_k \to \bar{w}$ (for example, $\bar{w}_k = \nabla \ell(x_k)$ for a $C^1$ loss $\ell$), then for almost all $\theta \in \RR^m$, $\bar{\theta}_k^T \to \bar{w}^T  \frac{\partial \bar{x}}{\partial \theta}$.    \label{prop:convergenceAlogirthmicDiff}
\end{proposition}
\begin{remark}{\rm
				In addition to \Cref{prop:convergenceAlogirthmicDiff}, in both cases, for all $\theta$, all accumulation points of both $\dot{x}_k$ and $\bar{\theta}_k^T$ are elements of $\jalg \dot{\theta}$ and $\bar{w}^T\jalg$ respectively.
				\label{rem:convergenceAutodiff}}
\end{remark}
\subsection{Linear convergence rate for semialgebraic piecewise smooth selection function}
Semialgebraic functions are ubiquitous in machine learning (piecewise polynomials, $\ell_1$, $\ell_2$ norms, determinant matrix rank \ldots). We refer the reader to \cite{bolte2020mathematical} for a thorough discussion of their extensions, and use in machine learning. For more technical details, see \cite{coste2000introduction,coste2000introduction2} for introductory material on semialgebraic and o-minimal geometry.
\paragraph{Lipschitz gradient selection functions.}
Let $F \colon \RR^p \mapsto \RR^q$ be semialgebraic and continuous. We say that $F$ has a {\em Lipschitz gradient selection} $(s, F_1, \ldots, F_m)$ if $s \colon \RR^p \mapsto \left(1,\ldots, m\right)$ is semialgebraic and there exists $L\geq0$ such that for $i = 1 \ldots, m$, $F_i\colon \RR^p \mapsto \RR^p$ is semialgebraic with $L$-Lipschitz Jacobian, and for all $x \in \RR^p$, $F(x) = F_{s(x)}(x)$.

For any $x \in \RR^p$, set $I(x) = \left\{ i\in\left\{ 1,\ldots,m \right\}, \, F(x) = F_i(x) \right\}$. The set-valued map $J^s_F \colon \RR^p \rightrightarrows \RR^{p \times q}$ given by
\begin{align*}
  J^s_F \colon x &\rightrightarrows \mathrm{conv}\left( \left\{ \frac{\partial F_i}{\partial x}(x),\, i \in I(x) \right\} \right),
\end{align*}
is a conservative Jacobian for $F$ as shown in~\cite{bolte2020mathematical}. Here $\frac{\partial F_i}{\partial x}$ denotes the classical Jacobian of $F_i$. Let us stress that such a structure is ubiquitous in applications \cite{bolte2020mathematical,lee2020oncorrectness}.

\paragraph{Rate of convergence.}
We may now strengthen Corollary~\ref{cor:unrollingConvergence} by proving  the linear convergence of piggyback derivatives towards the fixed point. The following  is a consequence of the fact that the proposed selection conservative Jacobians of Lipschitz gradient selection functions are Lipschitz-like (\Cref{lem:explicitSelection} in \Cref{app:lipcons}). Note that semialgebraicity is only used as a \emph{sufficient} condition to ensure conservativity of the selection Jacobian together with this Lipschitz like property. It could be relaxed if it can be guaranteed by other means, in particular one could consider the broader class of definable functions in order to handle log-likelihood data fitting terms.
\begin{corollary}[Linear convergence of piggyback derivatives]\label{cor:lipconv}
				In addition to \Cref{ass:contraction}, assume that $F$ has a Lipschitz gradient selection structure as above. Then, for any $\theta$ and $\epsilon > 0$, there exists $C>0$ such that the recursion~\eqref{eq:autodiff} with $J_{F} = J_{F}^{s}$ satisfies
  \begin{align*}
    \gap( J_{x_k}(\theta),\jalg(\theta)) \leq C (\sqrt{\rho} + \epsilon)^k,\quad \forall k \in \NN.
  \end{align*}
Moreover, classical Jacobians in \Cref{cor:classicalJacobians} converge at a linear rate for almost all $\theta$.
\end{corollary}
\section{Application to proximal splitting methods in convex optimization}\label{sec:prox}
Consider the composite parametric convex optimization problem, where $\theta \in \RR^m$ represents parameters and $x \in \RR^p$ is the decision variable
\begin{align*}
    \bar{x}(\theta) = {\arg\min}_x f(x,\theta) + g(x,\theta).
\end{align*}
The purpose of this section is to construct examples of function $F$ used in recursion \eqref{eq:iterative} based on known algorithms. The following assumption will be standing throughout the section.
\begin{assumption}\label{ass:reg-prox}{\rm
				$f$ is semialgebraic, convex, its gradient with respecto to $x$ for fixed $\theta$, $\nabla_x f$, is locally Lipschitz jointly in $(x,\theta)$ and $L$-Lipschitz in $x$ for fixed $\theta$. Semialgebraicity implies that $\nabla_x f$ is path-differentiable jointly in $(x,\theta)$, we denote by $J^2_f$ its Clarke Jacobian.

				$g$ is semialgebraic, convex in $x$ for fixed $\theta$, and lower semicontinuous. For all $\alpha>0$, we assume that $G_\alpha (x,\theta) \mapsto \prox_{\alpha g(\cdot,\theta)}(x)$ is locally Lipschitz jointly in $(x,\theta)$. semialgebraicity implies  that it is also path differentiable jointly in $(x,\theta)$, we denote by $J_{G_\alpha}$ its Clarke Jacobian.}
\end{assumption}
This assumption covers a very large diversity of problems in convex optimization as most gradient and prox operations used in practice are semialgebraic.
Under Assumption \ref{ass:reg-prox}, we will provide sufficient conditions on $f$ and $g$ for \Cref{ass:contraction} to hold for different algorithmic recursions. These are therefore sufficient for the validity of the convergence results in  \Cref{cor:unrollingConvergence} and \Cref{cor:classicalJacobians}, \Cref{prop:convergenceAlogirthmicDiff}, as well \Cref{cor:lipconv} in the piecewise selection case. The proofs are postponed to \Cref{app:prox}.

\subsection{Splitting algorithms}
In this section we provide sufficient condition for \Cref{ass:contraction} to hold. The underlying conservative Jacobian is obtained by combining Clarke Jacobians of elementary algorithmic operations (gradient, proximal operator in \Cref{ass:reg-prox}), using the compositional rules of differential calculus \cite{bolte2020mathematical} and implicit differentiation \cite{bolte2021nonsmooth}. Using \cite{bolte2020conservative}, such Jacobians are conservative by semialgebraicity and their combination provide conservative Jacobians for the corresponding algorithmic recursion $F$. These objects are explicitly constructed in \Cref{app:prox}.
\paragraph{Forward--backward algorithm.}
The forward--backward iterations are given for $\alpha > 0$ by
\begin{align}
    \label{eq:FBalgo}
    x_{k+1} = \prox_{\alpha g(\cdot,\theta)}\left(x_k  - \alpha \nabla_x f(x_k,\theta)\right) .
\end{align}

\begin{proposition}\label{prop:fb}
				Under \Cref{ass:reg-prox} with $0 < \alpha < \frac{2}{L}$, denote by $F_\alpha \colon \RR^{p \times m} \to \RR^p$ the forward-backward recursion in \eqref{eq:FBalgo}. For $\mu > 0$, if either $f$ or $g$ is $\mu$-strongly convex in $x$ for all $\theta$, then $F_{\alpha}$ is a strict contraction and \Cref{ass:contraction} holds.
\end{proposition}

\paragraph{Douglas--Rachford.}
Given $\alpha > 0$, the algorithm goes as follows
\begin{align}
    \label{eq:DRalgo}
    y_{k+1} = \frac{1}{2}(I + R_{\alpha f(\cdot,\theta)} R_{\alpha g(\cdot,\theta)}) y_k,
\end{align}
where $R_{\alpha f(\cdot,\theta)} = 2 \prox_{\alpha f(\cdot,\theta)} - I$ is the reflected proximal operator, which is $1$-Lipschitz (and similarly for $g$).
Following \cite[Theorem 26.11]{bauschke2011convex}, if the problem has a minimizer, then $(y_k)_{k \in \NN}$ converges to a fixed point of \eqref{eq:DRalgo}, $\bar{y}$ such that $\bar{x} = \prox_{\alpha g}(\bar{y})$ is a solution to the optimization problem.
Following \cite[Theorem 1]{giselsson2016linear}, if $f$ is strongly convex, then $R_{\alpha f(\cdot,\theta)}$ is $\rho$-Lipschitz for some $\rho <1$ and our differentiation result applies to Douglas-Rachford splitting in this setting.

\begin{proposition}\label{prop:DR}
  Under \Cref{ass:reg-prox} with $\alpha > 0$, denote by $F_\alpha \colon \RR^{p \times m} \to \RR^p$ the Douglas-Rachford recursion in \eqref{eq:DRalgo}. If $f$ is $\mu$-strongly convex in $x$ for all $\theta$, then $F_{\alpha}$ is a strict contraction and \Cref{ass:contraction} holds.
\end{proposition}
\paragraph{Alternating Direction Method of Multipliers algorithms.}
Consider the separable convex problem
\begin{equation}\label{eq:sep-pb}
\min_{u, v} \phi_\theta(u) + \psi_{\theta}(v) \quad \text{subject to} \quad A_{\theta}u + B_{\theta}v = c_{\theta} .
\end{equation}
The alternating direction method of multipliers (ADMM) algorithm combines two partial minimization of an augmented Lagrangian, and a dual update:
\begin{equation}\label{eq:admm}
  \begin{split}
  u_{k+1} &= \arg\min_{u} \left\{ \phi_{\theta}(u) + x^{\top} A_{\theta} u + \frac{\alpha}{2} \| A_{\theta}u + B_{\theta}v_{k} - c_{\theta} \|_{2}^{2}  \right\} \\
  v_{k+1} &= \arg\min_{v} \left\{ \psi_{\theta}(v) + x^{\top} B_{\theta} v + \frac{\alpha}{2} \| A_{\theta}u_{k+1} + B_{\theta}v_{k} - c_{\theta} \|_{2}^{2}  \right\} \\
  x_{k+1} &= x_{k} + \alpha (A_{\theta}u_{k+1} + B_{\theta}v_{k+1} - c_{\theta}) .
  \end{split}
\end{equation}
As observed in~\cite{gabay1983chapter}, the ADMM algorithm can be seen as the Douglas-Rachford splitting method applied to the Fenchel dual of problem \eqref{eq:sep-pb} (see Appendix~\ref{app:eq-admm-dr} for more details). More precisely, ADMM updates are equivalent to Douglas-Rachford iterations applied to the following problem
\begin{equation}\label{eq:sep-pb-dual}
  \min_x \underbrace{c_{\theta}^\top x + \phi_{\theta}^*(-A_{\theta}^\top x)}_{f(x, \theta)} + \underbrace{\psi_{\theta}^*(-B_{\theta}^\top x)}_{g(x, \theta)}.
\end{equation}
Therefore, if $\phi_{\theta}$ is strongly convex with Lipschitz gradient and $A_{\theta}$ is injective, then ADMM converges linearly and one is able to combine derivatives of proximal operators to differentiate ADMM.

\subsection{Numerical illustration.}
We now detail how \Cref{fig:linearCvgt} discussed in the introduction is obtained, and how it illustrates our theoretical results.
We consider four scenarios (Ridge, Lasso, Sparse inverse covariance selection and trend filtering) corresponding to the four columns.
For each of them, the first line shows the empirical linear rate of the iterates $x_k$ and the second line shows the empirical linear rate of the derivative $\frac{\partial}{\partial \theta}x_k$.
All experiments are repeated 100 times and we report the median along with first and last deciles.

\paragraph{Forward--Backward for the Ridge.}
The Ridge estimator is defined for $\theta > 0$ as $\bar{x}(\theta) = \arg\min_{x \in \mathbb{R}^{p}} \frac{1}{2} \|Ax - b\|_2^2 + \theta \| x \|_{2}^{2}$
Among several possibilities to solve it, one can use the Forward--Backward algorithm applied to $f\colon (x,\theta) \mapsto \frac{1}{2} \|Ax - b\|_2^2$ and $g: \theta \| x \|_{2}^{2}$.
Since $g$ is strongly convex, the operator $F_{\alpha}$ is strongly convex, and thus \Cref{prop:fb} may be applied.

\paragraph{Forward--Backward algorithm for the Lasso.}
Consider the Forward--Backward algorithm applied to the Lasso problem \cite{tibshirani1996regression}, with parameter $\theta  > 0$, $\bar{x}(\theta) \in \arg\min_{x \in \mathbb{R}^{p}} \frac{1}{2} \|Ax - b\|_2^2 + \theta \|x\|_1  = \arg\min_x \frac{1}{2L} \|Ax - b\|_2^2 + \frac{\theta}{L} \|x\|_1$,
where $L$ is any upper bound on the operator norm of $A^TA$.
The gradient of the quadratic part is $1$ Lipschitz so we may consider the forward backward algorithm \eqref{eq:FBalgo}, with unit step size with $f\colon (x,\theta) \mapsto \frac{1}{2L} \|Ax - b\|_2^2$ and $g \colon (x,\theta) \mapsto  \frac{\theta}{L} \|x\|_1$.

A well known qualification condition involving a generalized support at optimality ensures uniqueness of the Lasso solution \cite{efron2004least,mairal2012complexity}. This conditions holds for generic problem data \cite{tibshirani2013lasso}.
Following \cite[Proposition 5]{bolte2021nonsmooth}, under this qualification condition, the implicit conservative Jacobian $J_F$ is such that, at the solution $x^*$, $J_F(x^*)$ has an operator norm of at most $1$, and the matrix set $I - J_F$ only contains invertible matrices. This means that there exists $\rho<1$, such that any $M \in J_F(x^*)$ has operator norm at most $\rho$. Following \Cref{rem:local}, all our convergence results apply qualitatively.
Note that we recover the results of~\cite[Proposition 2]{bertrand2020implicit} for the Lasso.

\paragraph{Douglas--Rachford for the Sparse Inverse Covariance Selection.}
The Sparse Inverse Covariance Selection~\cite{wainwright2006high,friedman2007sparse} reads $\bar x(\theta) \in \arg\min_{x \in \mathbb{R}^{n \times n}}\operatorname{tr}(Cx) - \log\det x + \theta \sum_{i,j} |x_{i,j}|$,
where $C$ is a symmetric positive matrix and $\theta > 0$.
It is possible to apply Douglas--Rachford to $f: (x,\theta) \mapsto \operatorname{tr}(Cx) - \log\det x$ and $g: (x, \theta) \mapsto \theta \| x \|_{1,1}$.
It is known that $f$ is locally strongly convex, indeed $x \mapsto - \log \det x$ is a standard self-concordant barrier in semidefinite programming~\cite{nesterov1994interior}.
Following \Cref{rem:local}, all our convergence results apply qualitatively.

\paragraph{ADMM for Trend Filtering.}
Introduced in~\cite{tibshirani2014trend} in statistics as a generalization of the Total Variation, the trend filtering estimator with observation $\theta \in \mathbb{R}^p$ reads $\bar x(\theta) = \arg\min_{x \in \mathbb{R}^{p}} \frac{1}{2} \| x - \theta \|_2^2 + \lambda \| D^{(k)} x \|_1$,
where $D^{(k)}$ is a forward finite--difference approximation of a differential operator of order $k$ (here $k=2$).
Using $\psi_{\theta}: u \mapsto \lambda \| u \|_{1}$ (strongly convex), $\phi_{\theta}: v \mapsto \| v - \theta \|_{2}^{2}$, $A_{\theta} = - I$ (injective), $B_{\theta} = D^{(k)}$, and $c_{\theta}=0$, we can apply the ADMM to solve trend filtering.

\section{Failure of automatic differentiation for inertial methods}
\label{sec:inertial}
\begin{figure}[b]
  \centering
  \includegraphics[width=\textwidth]{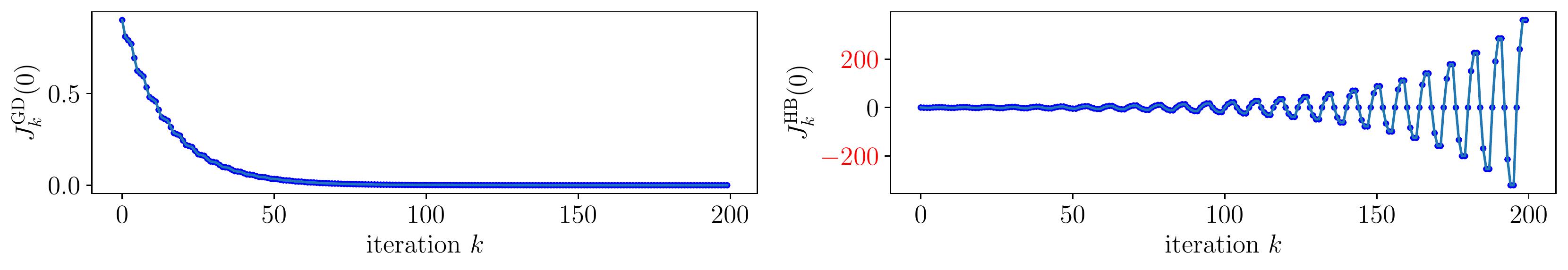}
  \caption{Behaviour of automatic differentiation for first-order methods on a quadratic function. (Left) Stability of the propagation of derivatives for the fixed step-size gradient descent. (Right) Instability of the propagation of Heavy-Ball initialized. Both methods are initialized at optimum.}
  \label{fig:heavyBallInstability}
\end{figure}
In this section we consider the Heavy-Ball method for strongly convex objectives, in its global linear convergence regime. When applied to a $C^2$ objective, accumulation of derivatives converges to the derivative of the solution map~\cite{griewank2003piggyback,mehmood2020automatic,lorraine2020optimizing}. However, we provide a $C^{1,1}$ strongly convex parametric objective with path differentiable derivative, such that forward propagation of derivatives along the Heavy Ball algorithm contains diverging vectors for a given parameter value.
In this example, one may obtain a conservative Jacobian by other means, such as implicit differentiation or algorithmic differentiation of the gradient descent algorithm, both avoiding this divergent behaviors.

\subsection{Heavy-ball algorithm and global convergence}
Consider a function $f \colon \RR^{p} \times \RR^{m} \to \RR$, and $\beta > 0$, for simplicity, when the second argument is fixed we write $f_\theta \colon x \mapsto f(x,\theta)$. Set for all $x,y,\theta$, $F(x,y,\theta) = (x - \nabla f_\theta(x) + \beta (x - y), x)$, consider the Heavy-Ball algorithm $(x_{k+1}, y_{k+1}) = F(x_k, y_k, \theta)$ for $k \in \NN$.

 If $f_\theta$ is $\mu$-strongly convex with $L$-Lipschitz gradient, then, choosing $\alpha = 1/L$ and $\beta < \frac{1}{2}\left( \frac{\mu}{2L} + \sqrt{\frac{\mu^2}{4L^2} + 2} \right)$,
the algorithm will converge globally at a linear rate to the unique solution, $\bar{x}(\theta)$ \cite[Theorem 4]{ghadimi2015global}, local convergence is due to Polyak \cite{polyak1987introduction}. Furthermore, if in addition $f$ is $C^2$ forward propagation of derivatives converge to the derivative of the solution \cite{griewank2003piggyback,griewank2008evaluating,mehmood2020automatic}.
\subsection{A diverging Jacobian accumulation}
Details and proof of the following result are given in Section \ref{sec:suppInertial}.
\begin{proposition}[Piggyback differentiation  fails for the Heavy Ball method]\label{prop:hb-fail}
  Consider $f \colon \RR^2 \to \RR$, such that for all $\theta \in \RR$, $f(x,\theta) = x^2/2$ if  $x \geq 0$ and  $f(x,\theta) = x^2/8$ if  $x < 0$. Assume that $\alpha = 1$ and $\beta = 3/4$. Then the heavy ball algorithm  converges globally to $0$ and $\nabla f$ is path differentiable. The Clarke Jacobian of $F$ with respect to $(x,y)$ at $(0,0,0)$ is $J_F(0,0,0) =  \mathrm{conv} \left\{M_1,M_2 \right\}$,
  where the product $M_1M_1M_2M_2$ has eigenvalue $-9/8$.
  \label{prop:jacobianHeavyBx^2/2all}
\end{proposition}
The presence of an eigenvalue with modulus greater than $1$ may produce divergence in \eqref{eq:autodiff}. Set
\begin{align*}
    f_1 \colon (x,\theta)  \mapsto
    \begin{cases}
         x^2/2 & \text{ if } x \geq 0\\
         x^2/8 & \text{ if } x < 0.
    \end{cases}\qquad
    f_2 \colon (x,\theta)  \mapsto
    \begin{cases}
        x^2/2 & \text{ if } x > 0\\
        x^2/8 & \text{ if } x \leq 0.
    \end{cases}
\end{align*}
Note that $f_1$ and $f_2$ are both equivalent to $f$ as they implement the same function.
With initializations $x(\theta) = y(\theta) = \theta$, we run a few iterations of the Heavy Ball algorithm for $\theta = 0$, and implement \eqref{eq:autodiff} alternating between two steps on $f_1$ and two steps on $f_2$ and differentiate the resulting sequence $(x_k)_{k\in \NN}$ with respect to $\theta$ using algorithmic differentiation. The divergence phenomenon predicted by \Cref{prop:hb-fail} is illustrated in Figure \ref{fig:heavyBallInstability}, while the true derivative is $0$ (the sequence is constant).

\section{Conclusion}
We have developed a  flexible theoretical framework to describe convergence of piggyback differentiation applied to nonsmooth recursions -- providing, in particular, a rigorous  meaning to  the differentiation of nonsmooth solvers. The relevance of our approach is illustrated on some major composite convex optimization problems through  widely used methods as forward-backward, Douglas-Rachford or ADMM algorithms. Our framework allows however to consider many other abstract algorithmic recursions and provides thus theoretical ground for more general problems such as variational inequalities or saddle point problems as in \cite{chambolle2021learning,bogensperger2022convergence}. As a matter for future work, we shall consider relaxing \Cref{ass:contraction} to study a wider class of methods, e.g., when $F$ is not a strict contraction.

\section*{Acknowledgments}
J. B. and E. P. acknowledge the financial support of the AI Interdisciplinary Institute ANITI funding under the grant agreement ANR-19-PI3A-0004, Air Force Office of Scientific Research, Air Force Material Command, USAF, under grant numbers FA9550-19-1-7026, and ANR MaSDOL 19-CE23-0017-01.
J. B. also acknowledges the support of ANR Chess, grant ANR-17-EURE-0010, TSE-P and the Centre Lagrange.
S. V. acknowledges the support ANR GraVa, grant ANR-18-CE40-0005.

\bibliographystyle{plain}
\bibliography{references}
\newpage
\appendix

This is the appendix for ``Convergence of piggyback differentiation of nonsmooth iterative solvers''.
\etocdepthtag.toc{mtappendix}
\etocsettagdepth{mtsection}{none}
\etocsettagdepth{mtappendix}{section}
\tableofcontents

\section{Properties of affine iterations on compact subsets}
\label{app:prop-aff-iter}

\subsection{Banach--Picard theorem: Proof of~\Cref{th:convergenceRecursion}}
\label{app:proof-fixed-point}
For a compact set, $\Zc$ we denote by $\|\Zc\|_{\sup}$ the maximal norm of elements in $\Zc$:
\begin{align*}
   \|\Zc\|_{\sup} = \sup_{z \in \Zc} \|z\|.
\end{align*}
In order to prove our fixed point result, we need first the following lemma.
\begin{lemma}[Bounding Hausdorff distances]
  Let $\Xc,\Yc, \Zc \subset \RR^p$ be nonempty compact sets, such that $\Xc \subset \Yc + \Zc$ and $\Yc \subset  \Xc + \Zc$ then
  \begin{align*}
    \dist(\Xc,\Yc) \leq \|\Zc\|_{\sup}
  \end{align*}
  \label{lem:hausdorffDifference}
\end{lemma}
\begin{proof}
  The first inclusion says that for any $x \in \Xc$, there is $y(x) \in  \Yc$, $z(x) \in \Zc$ such that $x = y(x)+z(x)$. We deduce that for any $x \in \Xc$
  \begin{align*}
    \min_{y \in \Yc} \|x- y\| &= \min_{y \in \Yc} \|y(x) - z(x) - y\| \leq \|z(x)\| \leq  \max_{z \in \Zc} \|z\|\\
  \end{align*}
  Therefore, $\max_{x \in \Xc} \min_{y \in \Yc} \|x - y\| \leq \max_{z \in \Zc} \|z\|$. Symmetrically, $\max_{y \in \Yc} \min_{x \in \Xc} \|x - y\| \leq \max_{z \in \Zc} \|z\|$ and the result follows.
\end{proof}

We now prove~\Cref{th:convergenceRecursion}.

\begin{proof}[of~\Cref{th:convergenceRecursion}]
  Recall that the action of $\Jc$ on matrices is defined in \eqref{eq:actionJMatrix} and by $\Ac$ and $\Bc$ the projections of $\Jc$ on the first $p$ and last $l$ columns respectively, that is $\Ac = \{A,\, \exists B, \, [A,B] \in \Jc\}$ and similarly for $B$. Note that $\Ac$ is a compact set and that all matrices in $\Ac$ have an operator norm of at most $\rho$. We claim that the restriction of $\Jc$ to compact subsets is a strict contraction in Hausdorff metric. Indeed, for any $\Xc$, $\Yc$ compact subsets of $\RR^{p \times m}$, we have by using Lemma~\ref{lem:hausdorffDifference} and noting that $\Jc$ preserves the inclusion,
  \begin{align*}
    \Jc(\Xc) \subset \Jc(\Yc + \dist(\Xc, \Yc)\BB) \subset \Jc(\Yc) + {\dist(\Xc, \Yc)}\Ac \BB \subset\Jc(\Yc) + \rho{\dist(\Xc, \Yc)}\BB \\
    \Jc(\Yc) \subset \Jc(\Xc + \dist(\Xc, \Yc)\BB) \subset \Jc(\Xc) + {\dist(\Xc, \Yc)}\Ac \BB \subset\Jc(\Xc) + \rho{\dist(\Xc, \Yc)}\BB
  \end{align*}
  where the last inclusion follows because $\Ac \BB \subset \rho \BB$, where $\BB$ is the unit ball (for the Euclidean norm) of $p \times m$ matrices, since by assumption all square matrices in $\Ac$ have operator norm at most $\rho$.
  We deduce that $\dist(\Jc(\Xc), \Jc(\Yc)) \leq \rho \dist(\Xc, \Yc)$ using Lemma \ref{lem:hausdorffDifference}, that is the action of $\Jc$ on subsets of $p\times m$ matrices is $\rho$ Lipschitz with respect to Hausdorff metric.

  Let us show that $(\Xc_k)_{k \in \NN}$ remains in a bounded set, we have for all $k$
  \begin{align*}
    \|\Xc_{k+1}\|_{\sup} \leq \|\Ac\Xc_{k} + \Bc\|_{\sup} \leq \|\Ac\Xc_{k}\|_{\sup} + \|\Bc\|_{\sup} \leq  \rho \|\Xc_{k}\|_{\sup} + \|\Bc\|_{\sup},
  \end{align*}
  which entails
  \begin{align*}
    \|\Xc_{k+1}\|_{\sup} - \frac{\|\Bc\|_{\sup}}{1 - \rho}\leq  \rho\left( \|\Xc_k\|_{\sup} - \frac{\|\Bc\|_{\sup}}{1 - \rho} \right).
  \end{align*}
  We distinguish two cases
  \begin{itemize}
  \item if $\|\Xc_{k}\|_{\sup} > \frac{\|\Bc\|_{\sup}}{1 - \rho}$, then  $\|\Xc_{k+1}\|_{\sup}$ gets either closer to $\frac{\|\Bc\|_{\sup}}{1 - \rho}$ or below it, in particular it decreases.
  \item if $\|\Xc_{k}\|_{\sup} \leq \frac{\|\Bc\|_{\sup}}{1 - \rho}$ then $\|\Xc_{k+1}\|_{\sup}\leq \frac{\|\Bc\|_{\sup}}{1 - \rho}$ and we remain below this threshold for all $k$.

  \end{itemize}
  All in all, for all $k \in \NN$,
  \begin{align*}
    \|\Xc_{k+1}\|_{\sup} \leq \max\left\{\|\Xc_{k}\|_{\sup},\frac{\|\Bc\|_{\sup}}{1 - \rho} \right\} \leq \ldots \leq  \max\left\{\|\Xc_{0}\|_{\sup},\frac{\|\Bc\|_{\sup}}{1 - \rho} \right\},
  \end{align*}
  and $\lim\sup_k \|\Xc_{k}\|_{\sup} \leq \frac{\|\Bc\|_{\sup}}{1 - \rho}$.

  We have shown that the sequence remains in a bounded set so that the recursion actually takes place in a compact set $\Cc \subset \RR^{p \times m}$ which contains all the iterates in its interior, we consider the restriction of the topology to this subset. By \cite[Theorem 3.85]{aliprantis2006infinite}, the closed subsets form a complete metric space. The result is an application of Banach-Picard theorem (for example \cite[Section 10.3]{royden1988real}). In particular (see \cite[Theorem 3.88]{aliprantis2006infinite}), $\Lc$ is the unique fixed point of $\Jc$ and it is closed and bounded, hence compact. Note that we can consider larger compact sets to take into account larger initializations, the fixed point remains the same. Indeed for a larger compact $\tilde{\Cc}$ containing $\Cc$, $\Lc$ is in the interior of $\Cc$ and is still a fixed point of $\Jc$ when the topology is restricted to $\tilde{\Cc}$ and this fixed point must be unique.
\end{proof}

\subsection{Properties of the fixed-set mapping}
\label{app:prop-fix-map}
We now equip the set of matrices $\RR^{p \times (p \times m)}$ with the norm $\|[A,B]\|_{p,m} = \max\{\|A\|_{\mathrm{op}}, \|B\|\}$ where $A \in \RR^{p \times p}$ and $B \in \RR^{p \times m}$. The set of compact subsets of $\RR^{p \times (p +m)}$ is endowed with the corresponding Hausdorff distance.

\begin{definition}[Affine contraction sets]
    For $\rho\in[0,1)$, we denote by $\mathcal{C}_\rho$, the set of compact subsets of matrices in $\RR^{p \times (p +m)}$ such that for all $\mathcal{S} \subset \RR^{p \times (p+m)}$, $\mathcal{S} \subset \mathcal{C}_\rho$ and  all $M \in \mathcal{S}$, we have $\|A\|_{\mathrm{op}}\leq \rho$ where $A \in \RR^{p \times p}$ is the matrix made of the first $p$ columns of $M$.
    \label{def:affineContractionSet}
\end{definition}

Given $\Jc \in \mathcal{C}_\rho$, we denote by $\fixj$ the unique fixed point  of the corresponding iteration mapping as defined in Theorem \ref{th:convergenceRecursion}. We have the following
\begin{proposition}[Monotonicity of the fixed set]
  Given $\Jc \in \mathcal{C}_\rho$ and $\tilde{\Jc} \in \mathcal{C}_\rho$ (as in \Cref{def:affineContractionSet}),  such that $\Jc \subset \tilde{\Jc}$, we have
  \begin{align*}
    \fixj \subset \fix(\tilde{\Jc}).
  \end{align*}
  \label{prop:inclusion}
\end{proposition}
\begin{proof}
  Setting $\Xc_0 = \fix (\Jc)$, we have
  \begin{align*}
    \Xc_0 = \Jc(\Xc_0) \subset \tilde{\Jc}(\Xc_0),
  \end{align*}
  and the result follows by the same argument as in the last paragraph of the proof of Theorem \ref{th:convergenceRecursion}.
\end{proof}

\begin{proposition}[The fixed-set mapping is locally Lipschitz continuous]
  The function $\fix$ is locally Lipschitz continuous on $\mathcal{C}_\rho$ (as in \Cref{def:affineContractionSet}). More precisely, for any $\Jc_0 \in \mathcal{C}_\rho$ and $\Jc \in \Cc_\rho$,
  \begin{align*}
    \dist\left(\fix (\Jc_0),\fix (\Jc)\right) \leq \left(\frac{1}{1- \rho} + \frac{\sup_{[A_0,B_0] \in \Jc_0} \|B_0\| }{(1 - \rho)^2}\right) \dist(\Jc_0, \Jc)
  \end{align*}
  \label{prop:lipschitzContinuity}
\end{proposition}
\begin{proof}
  Given $\Jc_0 \in \mathcal{C}_\rho$ and $\Jc\in \mathcal{C}_\rho$, we remark that $\Jc \subset \Jc_0 + \dist(\Jc_0, \Jc) \mathbb{B}_{pm}$, where $\dist$ and $\mathbb{B}_{pm}$ are considered with respect to the norm $\|\cdot\|_{pm}$. This means
  \begin{align*}
    \Jc \subset \left\{ [A_0,B_0] + [C,D],\, [A_0,B_0]\in \Jc_0, \|[C,D]\|_{p,m} \leq \dist(\Jc_0, \Jc) \right\}
  \end{align*}
  We have
  \begin{align*}
    \Jc(\fix (\Jc_0))  =\;& \left\{AX + B, \, [A,B] \in \Jc, \, X \in \fix (\Jc_0) \right\} \\
    \subset\;& \left\{A_0X + B_0, \, [A_0,B_0] \in \Jc_0, \, X \in \fix (\Jc_0) \right\} \\
                          &+  \left\{CX + D, \, \|[C,D]\|_{mp} \leq \dist(\Jc_0, \Jc) , \, X \in \fix (\Jc_0)\right\}\\
    =\;& {\Jc_0} (\fix (\Jc_0)) +  \left\{CX + D, \, \|[C,D]\|_{mp} \leq \dist(\Jc_0, \Jc) , \, X \in \fix (\Jc_0)\right\}\\
    =\;& \fix (\Jc_0) +  \left\{CX + D, \, \|[C,D]\|_{mp} \leq \dist(\Jc_0, \Jc) , \, X \in \fix (\Jc_0)\right\}.
  \end{align*}
  This sets one inclusion. Similarly, we have
  \begin{align*}
    \fix (\Jc_0) =\;& {\Jc_0}(\fix (\Jc_0))\\
    \subset\;& {\Jc} (\fix (\Jc_0)) +  \left\{CX + D, \, \|[C,D]\|_{mp} \leq \dist(\Jc_0, \Jc) , \, X \in \fix (\Jc_0)\right\}.\\
  \end{align*}
  Recall that $\|[C,D]\|_{mp} = \max\{\|C\|_{\mathrm{op}}, \|D\|\}$, we have for any $[C,D]$ with $\|[C,D]\|_{mp} \leq \dist(\Jc_0, \Jc)$ and  $X \in \fix (\Jc_0)$,
  \begin{align*}
    \Vert CX + D\Vert \leq \|C\|_{\op} \Vert \fix (\Jc_0) \Vert_{\sup} + \|D\| \leq \dist(\Jc_0, \Jc) (1 + \Vert \fix (\Jc_0) \Vert_{\sup}).
  \end{align*}
  We deduce using Lemma \ref{lem:hausdorffDifference} that $\dist(\fix (\Jc_0), \Jc(\fix (\Jc_0 ))) \leq \dist(\Jc_0, \Jc)  (1 + \Vert \fix (\Jc_0) \Vert_{\sup})$. Setting $\Xc_0 = \fix (\Jc_0)$, invoking Theorem \ref{th:convergenceRecursion} with $\Jc$ and $k = 0$, we have
  \begin{align*}
    \dist(\fix (\Jc_0),\fix (\Jc)) &\leq \frac{\dist(\Jc_0, \Jc)  (1 + \Vert \fix (\Jc_0) \Vert_{\sup})}{1 - \rho} \\
                                   &\leq  \dist(\Jc_0, \Jc) \frac{ (1 - \rho + \sup_{[A_0,B_0] \in \Jc_0} \|B_0\| )}{(1 - \rho)^2} .
  \end{align*}
\end{proof}

\subsection{Perturbed iterations}
\label{app:perturbed-iter}
The following proposition shows that the linear convergence property is actually stable to perturbations. It will be useful to show that all potential limits of unrolling algorithmic differentiation recursions are contained in the corresponding fixed point set.
\begin{proposition}[Perturbed set sequences]
  Let $\rho < 1$ and $\epsilon > 0$ such that $\rho + \epsilon <1$.
  Let $(\Jc_k)_{k \in \NN}$ be a sequence in $\Cc_{\rho + \epsilon}$ and $\bar{\Jc} \in \Cc_\rho$ (as in \Cref{def:affineContractionSet}). Assume that for all $k \in \NN$
  \begin{align*}
    \gap_{pm} (\Jc_k,\bar{\Jc})  \leq \epsilon
  \end{align*}
  or in other words $\Jc_k \subset \bar{\Jc} + \epsilon \mathbb{B}_{pm}$ where $\mathbb{B}_{pm}$ is the unit ball of the norm $\|\cdot\|_{pm}$.
  Then the recursion on compact sets
  \begin{align*}
    \Xc_{k+1} = {\Jc_k} (\mathcal{X}_k)
  \end{align*}
  satisfies $ $ for all $k \in \NN$
  \begin{align*}
    &\gap(\Xc_k,\fixj) \\
    \leq\quad & (\rho+\epsilon)^{k} \frac{(1+\rho+\epsilon) \|\Xc_0\|_{\sup} + \sup_{[A,B] \in \bar{\Jc}} \|B\| + \epsilon}{1 - \rho- \epsilon} + \epsilon\frac{(1 - \rho+ \sup_{[A,B] \in \bar{\Jc}} \|B\|)}{(1 - \rho)^2}.
  \end{align*}
  In other words, $\Xc_k \subset \fix (\bar{\Jc}) + C(\rho, \epsilon, k) \BB$ where $C(\rho, \epsilon, k)$ is the constant above.
  \label{prop:perturbedConvergence}
\end{proposition}
\begin{proof}
  Set $\Jc_\epsilon := \{J + [C,D],\, J \in \bar{\Jc},\, \|[C,D]\|_{mp} \leq \epsilon\}$. Denote by $(\tilde{\Xc}_k)_{k\in\NN}$ the sequence satisfying the recursion, $\tilde{\Xc}_{k+1} = \Jc_\epsilon(\tilde{\Xc}_{k})$ with $\Xc_0 = \tilde{\Xc}_0$. We have
  \begin{align*}
    \Xc_1 =  \bar\Jc(\Xc_0) \subset \Jc_\epsilon(\Xc_0) = \tilde{\Xc}_1
  \end{align*}
  and by recursion $\Xc_k \subset \tilde{\Xc}_k$ for all $k \in \NN$. By Theorem \ref{th:convergenceRecursion}, we have
  \begin{align*}
    \dist(\tilde{\Xc}_k, \fix (\Jc_\epsilon))& \leq (\rho+\epsilon)^{k} \frac{\dist(\Xc_0, \Jc_\epsilon(\Xc_0))}{1 - \rho- \epsilon}.
  \end{align*}
  We deduce from Proposition \ref{prop:lipschitzContinuity} that for all $k \in \NN$,
  \begin{align*}
    &\dist(\tilde{\Xc}_k, \fix (\bar{\Jc}))\\
    \leq \quad & \dist(\tilde{\Xc}_k, \fix (\Jc_\epsilon)) + \dist(\fix (\Jc_\epsilon), \fix (\bar{\Jc}))\\
    \leq \quad & (\rho+\epsilon)^{k} \frac{\dist(\Xc_0, \Jc_\epsilon(\Xc_0))}{1 - \rho- \epsilon} + \frac{ (1 - \rho + \sup_{[A,B] \in \bar{\Jc}} \|B\| )}{(1 - \rho)^2} \dist(\Jc_\epsilon, \bar{\Jc}) \\
    \leq\quad & (\rho+\epsilon)^{k} \frac{(1+\rho+\epsilon) \|\Xc_0\|_{\sup} + \sup_{[A,B] \in \bar{\Jc}} \|B\| + \epsilon}{1 - \rho- \epsilon} + \frac{(1 - \rho+ \sup_{[A,B] \in \bar{\Jc}} \|B\|)}{(1 - \rho)^2}\,\epsilon.
  \end{align*}
  And the result follows because
  \begin{align*}
    \max_{X \in \Xc_k} \min_{L \in \fix (\bar{\Jc}) } \|X - L\| \leq \max_{X \in \tilde{\Xc}_k} \min_{L \in \fix (\bar{\Jc}) } \|X - L\|  \leq \dist(\tilde{\Xc}_k, \fix (\bar{\Jc})).
  \end{align*}
\end{proof}

This allows to obtain explicit convergence results as follows
\begin{corollary}[Limit of iterations with vanishing perturbations]
  Let $\rho < 1$ and $\bar{\Jc} \in \Cc_\rho$ (as in \Cref{def:affineContractionSet}).
  Let $(\Jc_k)_{k \in \NN}$ be a sequence of matrices such that for all $k \in \NN$
  \begin{align*}
    \gap_{pm}( \Jc_k,\bar{\Jc})  \leq \epsilon_k
  \end{align*}
  where $(\epsilon_k)_{k \in \NN}$ is a positive sequence such that there exists a constant $a>0$ such that $\epsilon_k \leq a\rho^{k}$ for all $k \in \NN$.
  Then for the recursion on compact sets of $p \times m$ matrices
  \begin{align*}
    \Xc_{k+1} = {\Jc_k} (\mathcal{X}_k)
  \end{align*}
  There are constants $C,c>0$ such that for all $k \in \NN$
  \begin{align*}
    &\gap(\Xc_k, \fixj )\leq Ce^{-ck}.
  \end{align*}
	Furthermore, one can take $c = \log\left(\frac{1}{\sqrt{ \rho + \epsilon}  }\right)$ for arbitrary $\epsilon>0$.
  \label{cor:perturbedConvergence}
\end{corollary}
\begin{proof}
				We consider $K\in \NN$ such that $\epsilon_k \leq \epsilon$ for all $k \in \NN$ where $\epsilon + \rho < 1$.
				Without loss of generality, we may assume that $K = 0$.
				Using the same notations as in the proof of Proposition \ref{prop:perturbedConvergence}, we have $\Xc_k \subset \tilde{\Xc}_k$ for all $k \in \NN$. Furthermore, it follows from the same arguments as in the proof of Theorem 1 that
  \begin{align}
    \|\Xc_k\|_{\sup} \leq \|\tilde{\Xc}_k\|_{\sup} \leq  M,
    \label{eq:perturbedConvervgenceTemp1}
	\end{align}
	for a constant $M > 0$.
  Now choose $k \in \NN$, applying Proposition \ref{prop:perturbedConvergence} shifting the initialization $0$ to $k$, we have for all $m \in \NN$
  \begin{align*}
    &\max_{X \in \Xc_{k+m}} \min_{L \in \fixj } \|X - L\| \\
    \leq\quad & (\rho+\epsilon_k)^{m} \frac{(1+\rho+\epsilon_k) \|\Xc_k\|_{\sup} + \sup_{[A,B] \in \bar{\Jc}} \|B\| + \epsilon_k}{1 - \rho- \epsilon_k} + \epsilon_k\frac{(1 - \rho+ \sup_{[A,B] \in \bar{\Jc}} \|B\|)}{(1 - \rho)^2}\\
    \leq\quad & (\rho+\epsilon)^{m} \frac{(1+\rho+\epsilon) M+ \sup_{[A,B] \in \bar{\Jc}} \|B\| + \epsilon}{1 - \rho- \epsilon} + a \rho^{k}\frac{(1 - \rho+ \sup_{[A,B] \in \bar{\Jc}} \|B\|)}{(1 - \rho)^2},
  \end{align*}
  where we have used the bound \eqref{eq:perturbedConvervgenceTemp1} and the fact that $\epsilon_k \leq \epsilon$ and $\epsilon_k \leq a \rho^k$.
  Setting $u = \frac{(1+\rho+\epsilon) M+ \sup_{[A,B] \in \bar{\Jc}} \|B\| + \epsilon}{1 - \rho- \epsilon}$ and $v = a \frac{(1 - \rho+ \sup_{[A,B] \in \bar{\Jc}} \|B\|)}{(1 - \rho)^2}$
  we have
  \begin{align*}
    \max_{X \in \Xc_{2k}} \min_{L \in \fixj } \|X - L\| \leq\quad & u(\rho+\epsilon)^{k}  + v \rho^{k} \leq (u+v)(\rho + \epsilon)^{2k/2} \leq \frac{u+v}{(\rho + \epsilon)^{1/2}}(\rho + \epsilon)^{2k/2},\\
    \max_{X \in \Xc_{2k+1}} \min_{L \in \fixj } \|X - L\| \leq\quad & u(\rho+\epsilon)^{k+1}  + v \rho^{k} \leq \frac{u+v}{(\rho + \epsilon)^{1/2}}(\rho + \epsilon)^{(2k+1)/2}.
  \end{align*}
  Since $k$ was arbitrary, this proves the desired result.
\end{proof}

\section{Existence of a conservative Jacobian for autodiff}
\label{app:cons-jac}

\subsection{Regularity of $\jalg$}
We recall the main notations and elements of \Cref{ass:contraction}.
We assume that $F$ is locally Lipschitz, path differentiable, and denote by $J_F\colon \RR^{p+m} \rightrightarrows \RR^{p \times (p+m)}$ a conservative Jacobian for $F$.
Now assume that any pair $[A,B] \in J_F(x,\theta)$ is such that the operator norm of $A$ is at most $\rho < 1$, that is for all $x$ and $\theta$, $J_F(x,\theta) \in \Cc_\rho$ (as in \Cref{def:affineContractionSet}). Define the following set-valued map
\begin{align*}
  \jalg \colon \theta \rightrightarrows \fix\left[
  J_F(\bar{x}(\theta), \theta)\right].
\end{align*}
Here, $\bar{x}(\theta) = \fix(F_\theta)$ is the unique fixed point of the algorithmic recursion so that we actually have
  \begin{align*}
    \jalg \colon \theta \rightrightarrows \fix\left[
    J_F(\fix(F_\theta), \theta)\right].
  \end{align*}

We have the following
\begin{lemma}[Regularity of $\jalg$]
  The mapping $\jalg$ is nonempty valued, locally bounded and has a closed graph.
  \label{lem:regularityLimit}
\end{lemma}
\begin{proof}
  The fact that $\jalg$ is locally bounded and non empty valued comes from the fact that $J_F$ is locally bounded with nonempty values and $\bar{x}$ is locally Lipschitz combined with Theorem \ref{th:convergenceRecursion}.

  By local Lipschitz continuity of $\bar{x}$ and the fact that $J_F$ has a closed graph, the set-valued map $\theta \rightrightarrows J_F(\bar{x}(\theta), \theta)$ also has a closed graph. By continuity of $\fixj$ with respect to the Hausdorff distance,  see Proposition~\ref{prop:lipschitzContinuity}, $\jalg$ has a closed graph.
\end{proof}

\subsection{Proof of \Cref{thm:jalgconvervative}}
\label{app:proofConservative}
\begin{proof}
Following \Cref{rem:implicitDiff}, we set
\begin{align*}
    \jimp\colon \theta
    \rightrightarrows
    \left\{(I - A)^{-1} B, [A,B] \in J_F(\bar{x}(\theta), \theta) \right\},
\end{align*}
a conservative Jacobian for $\bar{x}$ and $L_0 = \jimp$.
Now set by recursion for all $k \in \NN$
    \begin{align*}
        L_{k+1} \colon \theta \rightrightarrows J_F(\bar{x}(\theta), \theta)(L_k(\theta)).
    \end{align*}
    Recall that this means for all $\theta \in \RR^m$ and $k \in \NN$
    \begin{align*}
         L_{k+1}(\theta) = \{A L + B,\, [A,B] \in J_F(\bar{x}(\theta), \theta),\, L \in L_{k}(\theta)\}.
    \end{align*}
    Since $F(\bar{x}(\theta), \theta) = \bar{x}(\theta)$ for all $\theta$, $J_F$ is conservative for $F$ and $L_0$ is conservative for $\bar{x}$, we have by induction that for all $k \in \NN$, $L_k$ is conservative for $\bar{x}$.

    Fix $l \colon \RR^m \to \RR^m$ an arbitrary Borel measurable selection in $\jalg$, that is $l(\theta) \in \jalg(\theta)$ for all $\theta \in \RR^m$. Such a selection exist by \cite[Theorem 18.20]{aliprantis2006infinite} because $\jalg$ has a closed graph by \Cref{lem:regularityLimit}. Set for all $k \in \NN$ a measurable selection
    \begin{align*}
        l_k \colon \theta \to \arg\min_{z \in L_k(\theta)} \|z - l(\theta)\|.
    \end{align*}
    The function $(z,\theta) \to \|z - l(\theta)\|$ is Caratheodory (continuous in $z$, measurable in $\theta$), so such a selection exists (Aliprantis Theorem 18.19).
     By  \Cref{th:convergenceRecursion}, we have that $\dist(L_k(\theta),\jalg(\theta))$ tends to $0$ as $k$ grows, for all $\theta \in \RR^m$, where the convergence is in Hausdorff distance. Actually since all set-valued objects are locally bounded, the convergence occurs uniformly on every compact. This implies in particular that $l_k$ converges pointwise to $l$.

    Fix an absolutely continuous path $\gamma \colon [0,1] \to \RR^m$. We have for all $k \in \NN$, by conservativity,
    \begin{align*}
         \bar{x}(\gamma(1)) -  \bar{x}(\gamma(0)) &= \int_0^1 l_k(\gamma(t)) \dot{\gamma}(t) dt.
    \end{align*}
    Furthermore, $l_k \circ \gamma$ is measurable, converges pointwise to $l \circ \gamma$ and $l_k \circ \gamma$ can be uniformly bounded, let $K$ be such a bound. The integrable function $g \colon t \mapsto K\|\dot{\gamma}(t)\|$ dominates the integrand and $l_k \circ \gamma \times \dot{\gamma}$ converges pointwise to $l \circ \gamma \times \dot{\gamma}$. By the dominated convergence theorem (see \cite[Section 4.4]{royden1988real} ), we have
    \begin{align*}
         \bar{x}(\gamma(1)) -  \bar{x}(\gamma(0)) &= \int_0^1 l(\gamma(t)) \dot{\gamma}(t) dt.
    \end{align*}
    $L$ has a Castaing representation with a dense sequence of measurable selection \cite[Theorem 18.14]{aliprantis2006infinite}. Since $l$ was an arbitrary measurable selection in $L$, conservativity of $L$ follows by \cite[Lemma 8]{marx2022path}.
\end{proof}

\subsection{Proof of \Cref{cor:unrollingConvergence}}
\begin{proof}
   Fix $\theta$. We have $x_k(\theta) \to \bar{x}(\theta)$, so that for any $\epsilon > 0$, there exists $K \in \NN$ such that $J_F(x_k(\theta), \theta) \subset J_F(\bar{x}(\theta), \theta) + \epsilon \BB$ for all $k \geq K$. The result is then a consequence of Proposition \ref{prop:perturbedConvergence}, letting $\epsilon \to 0$. The last part is due to the conservativity of $\jalg$ which must be a singleton almost everywhere, equal to the classical Jacobian.
\end{proof}
\subsection{Proof of \Cref{cor:lipconv}}
\begin{proof}
	Define $(L_k)_{k \in \NN}$, a sequence of conservative Jacobians for $\bar{x}$ as in the begining of the proof of \Cref{thm:jalgconvervative} in \Cref{app:proofConservative}. By \cite[Theorem 1]{bolte2020conservative}, for each $k \in \NN$, there is a full measure set $S_k \subset \RR^m$ such that $L_k(\theta) = \left\{ \frac{\partial\bar{x} }{\partial \theta} (\theta) \right\}$ for all $\theta \in S_k$. Similarly, there exists a full measure set $S_{-1} \subset \RR^m$ such that $\jalg(\theta) = \left\{ \frac{\partial\bar{x} }{\partial \theta} (\theta) \right\}$ for all $\theta \in S_{-1}$. Setting $S = \cap_{i=-1}^{+\infty} S_i$, $S$ has full measure and for all $\theta \in S$ and for all $k \in \NN$,
	\begin{align*}
	    \jalg(\theta) = \left\{ \frac{\partial\bar{x} }{\partial \theta} (\theta) \right\} \qquad \qquad
	    L_k(\theta) = \left\{ \frac{\partial\bar{x} }{\partial \theta} (\theta) \right\}.
	\end{align*}
    Following the proof of \Cref{thm:jalgconvervative} in \Cref{app:proofConservative}, $L_k$ converges to $\jalg$ in Hausdorff distance, which means that convergence occurs in the classical sense since all sets in the sequence are singletons.
\end{proof}

\subsection{Proof of \Cref{prop:convergenceAlogirthmicDiff}}
\label{app:proofAutodiff}
\begin{proof}
				Under the setting of \Cref{cor:classicalJacobians},  for almost all $\theta \in \RR^m$, recursion \eqref{eq:autodiff} or \eqref{eq:autodiffSection} reduce to the following,  and all $k \in \NN$
    \begin{align}
        \label{eq:autodiffAppendixClassical}
        J_{k+1} = A_kJ_k + B_k
    \end{align}
		where $J_k = \frac{\partial x_k}{\partial \theta}$, $A_k = \frac{\partial F}{\partial x}(x_k,\theta)$ and $B_k = \frac{\partial F}{\partial \theta}(x_k,\theta)$ are classical Jacobians and $J_k$ converges to the classical Jacobian of $\frac{\partial \bar{x}}{\partial \theta}(\theta)$. Fix such a $\theta \in \RR^m$ and $k \in \NN$, $k \geq 1$.
    With the notation of Algorithm \ref{alg:autodiff0}, for the forward mode, multiplying \eqref{eq:autodiffAppendixClassical} on the right by $\dot{\theta}$, we have for all $i \in 1,\ldots k$
    \begin{align*}
						J_{i} \dot{\theta} = A_{i-1} J_{i-1} \dot{\theta} + B_{i-1} \dot{\theta}.
    \end{align*}
    Setting $\dot{x}_i = J_i \dot{\theta}$, this is exactly the recursion implemented by Algorithm \ref{alg:autodiff0} in forward mode. \Cref{cor:classicalJacobians} and the result follows from convergence of $J_k$.

     As for the backward mode a simple recursion shows that
    \begin{align}
				J_k =\quad & A_{k-1} A_{k-2} \ldots A_0 J_0 \nonumber\\
				+\:&A_{k-1} A_{k-2} \ldots A_1 B_0 \nonumber\\
        +\:& \ldots \nonumber\\
				+\:&A_{k-1} A_{k-2} \ldots A_i B_{i-1} \nonumber\\
        +\:& \ldots \nonumber\\
				+\:& A_{k-1} B_{k-2} \nonumber\\
				+\:& B_{k-1}.
    \end{align}
		Setting $B_{-1} = J_0$, we may rewrite equivalently,
		\begin{align}
						J_k = B_{k-1} + \sum_{i = 0}^{k-1} \left(\prod_{j = k-1}^{i} A_j\right) B_{i-1}.
						\label{eq:autodiffCompact}
		\end{align}
		Transposing and multiplying on the right by $\bar{w}_k$, we have
		\begin{align}
						J_k^T \bar{w}_k= B_{k-1}^T \bar{w}_k+ \sum_{i = 0}^{k-1} B_{i-1}^T \left(\prod_{j = {i}}^{k-1} A_j^T \right) \bar{w}_k.
						\label{eq:autodiffCompact2}
		\end{align}
		We set for all $i = 0, \ldots, k-1$,
		\begin{align}
						\bar{w}_{i} = \prod_{j = i}^{k-1} A_j^T \bar{w}_k.
						\label{eq:autodiffCompact3}
		\end{align}
		We have the backward recursion relation, for $i = k, \ldots, 1$
		\begin{align*}
						\bar{w}_{i-1} =  A_{i-1}^T \bar{w}_{i},
		\end{align*}
		which is the recursion implemented by Algorithm \ref{alg:autodiff0} in reverse mode. Combining \eqref{eq:autodiffCompact2} and \eqref{eq:autodiffCompact3}, we obtain
		\begin{align*}
						J_k^T \bar{w}_k= B_{k-1}^T\bar{w}_k + \sum_{i = 0}^{k-1} B_{i-1}\bar{w}_T = \sum_{i = 1}^{k} B_{i-1}^T\bar{w}_i + J_0^T \bar{w}_0,
		\end{align*}
		which is the quantity accumulated in $\bar{\theta}_k$ in Algorithm \ref{alg:autodiff0}. This proves that $\bar{\theta}_k^T$ returned by the backward mode is indeed equal to $\bar{w}_k^TJ_k$ and the convergence follows from convergence of both $\bar{w}_k$ and $J_k$ as $k \to \infty$.
\end{proof}

\section{Connection with implicit differentiation}
\label{app:impl-diff}
Recall that for all $\theta$
\begin{align*}{}
    \jimp(\theta) &=
    \left\{(I - A)^{-1} B, [A,B] \in J_F(\bar{x}(\theta), \theta) \right\} \\
     &=
    \left\{M,\, \exists [A,B] \in J_F(\bar{x}(\theta), \theta) \, M = AM + B\right\}.
\end{align*}
Setting $\mathcal{J} = J_F(\bar{x}(\theta), \theta)$,
we have therefore that $\jimp(\theta) \subset \J(\jimp(\theta))$. By recursion, for all $k \in \NN$, $\jimp(\theta) \subset \J^k(\jimp(\theta))$ and passing to the limit using Theorem \ref{th:convergenceRecursion}, $\jimp(\theta) \subset \fix(\J) = \jalg(\theta)$.
In particular, if $F$ is continuously differentiable, then \eqref{eq:autodiff} with classical Jacobians converges towards a classical implicit derivative.

However, the inclusion $\jimp(\theta) \subset \jalg(\theta)$ may be strict as the following example shows.
\begin{example}
    Set $\Jc = \{[A,B],\, A\in \Ac, B \in \Bc\}$, where
    \begin{align*}
      \Ac = \left\{
      \begin{pmatrix}
        \frac{\lambda + 1}{4}&0\\
        0& \frac{2 - \lambda}{4}
      \end{pmatrix},\, \lambda \in [0,1]
           \right\}
           \qquad
           \Bc = \left\{
           \begin{pmatrix}
             1\\
             1
           \end{pmatrix}
      \right\}.
    \end{align*}
    We set
    \begin{align*}
      \mathcal{T} = (I - \Ac)^{-1} \Bc =
      \left\{
      \begin{pmatrix}
        \frac{4}{3 - \lambda}\\
        \frac{4}{2 + \lambda}
      \end{pmatrix},\, \lambda \in [0,1]
      \right\}.
    \end{align*}

    As already observed, we have $\mathcal{T} \subset \Ac\mathcal{T} + \Bc$, but the inclusion is strict. Therefore $\mathcal{T}$ is not a fixed point of the affine iteration and it is only contained in it.

    Indeed, we have
    \begin{align*}
      \begin{pmatrix}
        \frac{1+1}{4}&0  \\
        0&\frac{2 - 1}{4}
      \end{pmatrix}
           \begin{pmatrix}
             \frac{4}{3 - 0}\\
             \frac{4}{2+0}
           \end{pmatrix} +
      \begin{pmatrix}
        1\\
        1
      \end{pmatrix} =
      \begin{pmatrix}
        \frac{5}{3}\\
        \frac{3}{2}
      \end{pmatrix}
      \in \Ac \mathcal{T} + \Bc.
    \end{align*}
    However solving for $\lambda$
    \begin{align*}
      \begin{pmatrix}
        \frac{5}{3}\\
        \frac{3}{2}
      \end{pmatrix} =
      \begin{pmatrix}
        \frac{4}{3 - \lambda}\\
        \frac{4}{2 + \lambda}
      \end{pmatrix},
    \end{align*}
    the first equation requires $\lambda = \frac{3}{5}$ while the second requires $\lambda = \frac{2}{3}$ which shows that the given vector does not belong to $\mathcal{T}$.
\end{example}

\section{Semialgebraic Lipschitz gradient selection functions}
\label{app:lipfun}

\subsection{Lipschitz property of conservative Jacobians of selections}
\label{app:lipcons}
\begin{lemma}[Conservative Jacobians of selections are Lipschitz-like]
  Let $F$ be continuous, semialgebraic with Lipschitz gradient selection. Then for each $x_0 \in \RR^p$, there exists $R>0$ such that
  \begin{align*}
    \gap(J^s_F(x), J^s_F(x_0))  &\leq L \|x - x_0\|, & \forall x,\, \|x - x_0\| \leq R,
  \end{align*}
  where $L$ is the Lipschitz constant given by the selection structure of $F$.
  \label{lem:explicitSelection}
\end{lemma}
\begin{proof}
  Fix $x_0 \in \RR^p$ and consider the function $g$ which associates to $r>0$ a subset of $\{1,\ldots, m\}$ defined as
  \begin{align*}
    g(r) = \cup_{\|x - x_0\| \leq r}\; I(x).
  \end{align*}
  The function $g$ is semialgebraic and therefore it admits a limit as $r \to 0$. The function $g$ is actually piecewise constant so that the limit is reached for some $R>0$ by semialgebraicity. This means that there is $R>0$ and an index set $I \subset \{1,\ldots, m\}$ such that $I(x) \subset I$ for all $x$ such that $\|x- x_0\| \leq R$. Furthermore, for each $i \in I$ and all $0<r\leq R$, there exists $x$ such that $\|x - x_0\| \leq r$ and $F_i(x) = F(x)$. By continuity of each component $F_i$, we have for each $i \in I$, $F_i(x_0) = F(x_0)$, that is $I \subset I(x_{0})$.

  We deduce that for each $x$ such that $\|x - x_0 \| \leq R$ and $i \in I(x)$, we have
  \begin{align*}
    \min_{V\in  J^s_F(x_0)} \left\Vert V - \frac{\partial F_i}{\partial x}(x)\right\Vert \leq  \left \Vert \frac{\partial F_i}{\partial x}(x_0) - \frac{\partial F_i}{\partial x}(x)\right\Vert \leq L\|x - x_0\|.
  \end{align*}
  Fix any $Z \in J^s_F(x)$, it is a convex combination of $\frac{\partial F_i}{\partial x}(x)$ for $i \in I(x)$ so by convexity of the distance, we have
  \begin{align*}
    \min_{V\in  J^s_F(x_0)} \|V - Z\| \leq L\|x - x_0\|,
  \end{align*}
  which proves the result since this allows to bound the supremum over $Z \in J^s_F(x)$ by the desired quantity.
\end{proof}

\subsection{Proof of \Cref{cor:lipconv}}
\begin{proof}
    This is a consequence of linear convergence of the recursion $x_{k+1} = F(x_k,\theta)$ combined with Lemma \ref{lem:explicitSelection} and Corollary \ref{cor:perturbedConvergence}.
\end{proof}

\section{Proximal splitting algortihms in convex optimization}
\label{app:prox}
\subsection{Proof of Proposition \ref{prop:fb}}
\label{app:proof-fb}
\begin{proof}
				We consider the gradient step operation $H_\alpha \colon (x,\theta) \mapsto x - \alpha \nabla_x f(x,\theta)$. We have for all $(x,\theta)$,
				\begin{align*}
								F_\alpha(x,\theta) = G_\alpha(H_\alpha(x,\theta),\theta).
				\end{align*}
				By \Cref{ass:reg-prox}, both $G_\alpha$ and $H_\alpha$ are $1$-Lipschitz in $x$ for fixed $\theta$ and we are going to show that if either $f$ or $g$ satisfy the strong convexity condition, the corresponding map is a strict contraction in $x$ for fixed $\theta$.
				Furthermore, the mapping $\jac^c_{H_\alpha} \colon (x,\theta) \rightrightarrows \left\{[I - \alpha A,- \alpha B],\, [A,B] \in J_f^2(x,\theta)\right\}$ is the Clarke Jacobian of $H_\alpha$.
				By \Cref{ass:reg-prox}, all the functions are path-differentiable \cite{bolte2020conservative} and one may obtain a conservative jacobian for $F$ by applying differential calculus rules \cite{bolte2020conservative}. We set for all $(x,\theta)$ a conservative Jacobian for $F_\alpha$,
				\begin{align}
								J_{F_\alpha}(x,\theta) = \left\{
												[C(I - \alpha A), -\alpha CB + D],\quad
								[A,B] \in J^2_{f}(x,\theta),\, [C,D] \in J_{G_\alpha}(x - \alpha \nabla_x f(x,\theta),\theta)\right\}
								\label{eq:jacobianfb}
				\end{align}
				Whenever $\nabla_x f$ is differentiable at $(x,\theta)$, the first $p$ columns of its Jacobian form a symmetric positive definite square matrix with eigenvalues at most $L$. This implies that the matrix $(I - \alpha A)$ in \eqref{eq:jacobianfb} is symmetirc with eignevalues in $[-1,1]$ and strictly greater than $-1$. Similarly, whenever $G_\alpha$ is differentiable, since it is $1$-Lipschitz in $x$ for fixed $\theta$ and the gradient of a $C^1$ function, the first $p$ columns of its Jacobian form a symmetric positive definite square matrix with eigenvalues at most $1$. This implies that the matrix $C$ in \eqref{eq:jacobianfb} is symmetric with eignevalues in $[0,1]$. In addition, we have the following;
				\begin{itemize}
								\item  Assume that for all $\theta$, $f$ is $\mu$-strongly convex. In this case, similarly as above the matrix $(I - \alpha A)$ in \eqref{eq:jacobianfb} has eigenvalue in $(-1,1)$ for all $(x,\theta)$.
								\item Assume that for all $\theta$, $g$ is $\mu$-strongly convex. In this case, similarly as above the matrix $C$ in \eqref{eq:jacobianfb} has eigenvalue in $[0,1/(1+\alpha \mu)]$ for all $(x,\theta)$ \cite[Proposition 23.13]{bauschke2011convex}.
				\end{itemize}
				In both cases, the product $C(I - \alpha A)$ in \eqref{eq:jacobianfb} has operator norm strictly smaller than $1$ and \Cref{ass:contraction} holds.
\end{proof}
\subsection{Proof of Proposition \ref{prop:DR}}
\label{app:proof-DR}

\begin{proof}
				From \cite[Proposition 23.11]{bauschke2011convex}, both $R_{\alpha f}$ and $R_{\alpha g}$ are $1$-Lipschitz. We are going to show that $R_{\alpha f}$ is a strict contraction and the result will follow. Since $f$ is $C^{1,1}$ in $x$, we have for all $\theta\in \RR^m$,
				\begin{align*}
								z = \prox_{\alpha f(\cdot,\theta)}(x) \Leftrightarrow  z+ \alpha \nabla_x f(z,\theta) - x = 0
				\end{align*}
				Set $H_\alpha(z,x,\theta) = z+ \alpha \nabla_x f(z,\theta) - x$, we have that
				\begin{align}
								\jac^c_{H_\alpha}(z,x,\theta) \rightrightarrows \left\{
												[I + \alpha A,\,
																-I,\,
												\alpha B]
								\right\}
								\label{eq:clarkeJacobianImplicit}
				\end{align}
				is the Clarke Jacobian of $H_\alpha$. Similarly as in \Cref{app:proof-fb}, by strong convexity of $f$, the matrix $I + \alpha A$ in \eqref{eq:clarkeJacobianImplicit} is symmetric with eigenvalues strictly greater than $0$ and smaller than $1$. By implicit differential calculus rule in \cite[Theorem 2]{bolte2021nonsmooth}, the mapping
				\begin{align}
								J_{\prox_{\alpha f(\cdot, \theta)}}(x,\theta) \rightrightarrows \left\{
												[(I + \alpha A)^{-1},\,-\alpha  (I + \alpha A)^{-1}B]
								,\, [A,B] \in J_f^2(\prox_{\alpha f(\cdot, \theta)},\theta)
								\right\}
								\label{eq:clarkeJacobianImplicit2}
				\end{align}
				is conservative for $(x,\theta) \mapsto \prox_{\alpha f(\cdot, \theta)}$. Furthermore, the matrix $(I + \alpha A)^{-1}$ in \eqref{eq:clarkeJacobianImplicit2} is symmetric eigenvalues in $(0,1)$. This entails that the mapping
				\begin{align}
								J_{R_{\alpha f(\cdot, \theta)}}(x,\theta) \rightrightarrows \left\{
												[2(I + \alpha A)^{-1} - I,\,-2\alpha  (I + \alpha A)^{-1}B - I]
								,\, [A,B] \in J_f^2(\prox_{\alpha f(\cdot, \theta)},\theta)
								\right\}
								\label{eq:clarkeJacobianImplicit3}
				\end{align}
				is conservative for $R_{\alpha f(\cdot, \theta)}$ and the matrix $2(I + \alpha A)^{-1} - I$ is symmetric with eigenvalues in $(-1,1)$.

				Similarly, the mapping
				\begin{align}
								J_{R_{\alpha g(\cdot, \theta)}}(x,\theta) \rightrightarrows \left\{
												[2 C - I,\,2D - I]
												,\, [C,D] \in J_{\prox_{\alpha g(x,\theta)}}
								\right\}
								\label{eq:clarkeJacobianImplicit4}
				\end{align}
				is the Clarke Jacobian of $R_{\alpha g(\cdot, \theta)}$ and the matrix $2 C - I$ in \eqref{eq:clarkeJacobianImplicit4} is symmetric with eigenvalues in $[-1,1]$. One may combine $J_{R_{\alpha f(\cdot, \theta)}}$ and $J_{R_{\alpha g\cdot, \theta)}}$,  using differential calculus rule to obtain a conservative Jacobian $J_{F_\alpha}$ for $F_\alpha$, such that for all $(x,\theta)$ and $[E,F] \in J_{F_\alpha}(x,\theta)$, the square matrix $E$ is of the form $\frac{I}{2} + ((I + \alpha A)^{-1} - I) (2C - I)$ where $A$ is from \eqref{eq:clarkeJacobianImplicit3} and $C$ is from \eqref{eq:clarkeJacobianImplicit4}. Such a matrix $E$ has operator norm strictly smaller than $1$ which is Assumption \ref{ass:contraction}.
\end{proof}

\subsection{Equivalence between ADMM and dual Douglas--Rachford}
\label{app:eq-admm-dr}

We need the following lemma.
\begin{lemma}\label{lem:dual-exp-prox}
    Let $F,G$ two convex, lower semicontinuous and closed functions and $h$ defined by
    \[ h(x) = F^*(-A^\top x) + G^*(x) . \]
    Then, $h$ is convex, lower semicontinuous, closed, and
    \begin{equation}
        \prox_{\alpha h}(x) = x + \alpha(A \hat u - \hat v)
    \end{equation}
    where
    \begin{equation*}
        (\hat u, \hat v) \in
        \arg\min_{u,v} \left\{
            F(u) + G(v) + x^\top (Au - v) + \frac{\alpha}{2} \| A u - v \|_2^2
        \right\} .
    \end{equation*}
\end{lemma}

The material contained in this section is already known in the litterature accross several papers and lecture notes, but for the sake of completeness, we include a full derivation of the equivalence.

In this appendix, we drop the dependency to the variable $\theta$ since we are only concerned on the behaviour with respect to $x$.
We recall that the iteration of Douglas--Rachford are defined by an initialization $y_{0}$ and the recursion
\begin{equation}\label{eq:dr-std}
\begin{split}
x_{k+1} &= \prox_{f}(y_{k}) \\
y_{k+1} &= y_{k} + \prox_{g}(2x_{k+1} - y_{k}) - x_{k+1} .
\end{split}
\end{equation}
By denoting $\tilde x_k = x_{k+1}$ and $\tilde y_k = y_k$, we can rewrite the updates of Douglas--Rachford (given $\tilde x_0$ and $\tilde y_0$) as
\begin{equation}\label{eq:dr-std-alter}
\begin{split}
\tilde y_{k+1} &= \tilde y_{k} + \prox_{g}(2\tilde x_{k} - \tilde y_{k}) - \tilde x_{k} . \\
\tilde x_{k+1} &= \prox_{f}(\tilde y_{k+1})
\end{split}
\end{equation}
Introducing the variable $\hat r = \prox_{g}(2\hat x - \hat y)$, this is also equivalent to
\begin{equation}\label{eq:dr-std-alter-bis}
\begin{split}
\hat r_{k+1} &= \prox_{g}(2\hat x_{k} - \hat y_{k}) \\
\hat x_{k+1} &= \prox_{f}(\hat y_{k} + \hat r_{k+1} - \hat x_{k}) \\
\hat y_{k+1} &= \hat y_{k} + \hat r_{k+1} - \hat x_{k}
\end{split}
\end{equation}
Using the change of variable $\hat w_{k} = \hat x_{k} - \hat y_{k}$, we have
\begin{equation}\label{eq:dr-std-alter-bis}
\begin{split}
\hat r_{k+1} &= \prox_{g}(\hat x_{k} + \hat w_{k}) \\
\hat x_{k+1} &= \prox_{f}(\hat r_{k+1} - \hat w_{k}) \\
\hat w_{k+1} &= \hat w_{k} + \hat x_{k+1} - \hat r_{k+1} .
\end{split}
\end{equation}
This formulation will be convenient to show how to retrieve the equations of ADMM~\eqref{eq:admm}.

The dual problem of \eqref{eq:sep-pb} is given by~\eqref{eq:sep-pb-dual}
\begin{equation}\label{eq:sep-dual}
\max_x - f(x) - g(x).
\end{equation}
where $f(x) = \phi^\star(-A x) + c^\top x$ and $g(x) = \psi(-B x)$

We consider the update rules given by \eqref{eq:dr-std-alter-bis}, i.e.,
\begin{align}
    \hat r &= \prox_{\alpha g}(x + w) \label{eq:update-r} \\
    \hat x &= \prox_{\alpha f}(\hat r - w) \label{eq:update-x} \\
    \hat w &= w + \hat x - \hat r .
\end{align}
Applying \Cref{lem:dual-exp-prox} to $F=\phi$ and $G=\iota_{c}$, we rewrite \eqref{eq:update-r} by
\[ \hat r = x + w + \alpha(A \hat u - c) \]
where
\[ \hat u = \arg\min_u \left\{ \phi(x) + x^\top ( Au - v ) + \frac{\alpha}{2} \| Au - c + w/\alpha \|_2^2 \right\} . \]
Using the same lemma to $F=\psi$ and $G=0$, we rewrite \eqref{eq:update-x} by
\[ \hat x = x + \alpha(A \hat u + B \hat v - c) \]
where
\[ \hat v = \arg\min_v \left\{ \psi(v) + x^\top Bv + \frac{\alpha}{2} \| A \hat u + B v - c \right\} .\]
Finaly, combining the expression of $\hat r$ and $\hat x$, we obtain
\[ \hat w = \alpha B \hat v . \]

\section{Inertial methods}
\label{sec:suppInertial}
Let us first recall notations from \Cref{sec:inertial}.
Consider a function $f \colon \RR^{p} \times \RR^{m} \to \RR$, and $\beta > 0$, for simplicity, when the second argument is fixed we write $f_\theta \colon x \mapsto f(x,\theta)$. Set for all $x,y,\theta$, $F(x,y,\theta) = (x - \nabla f_\theta(x) + \beta (x - y), x)$, consider the Heavy-Ball algorithm $(x_{k+1}, y_{k+1}) = F(x_k, y_k, \theta)$ for $k \in \NN$.
If $f_\theta$ is $\mu$-strongly convex with $L$-Lipschitz gradient, then, choosing $\alpha = 1/L$ and $\beta < \frac{1}{2}\left( \frac{\mu}{2L} + \sqrt{\frac{\mu^2}{4L^2} + 2} \right)$,
the algorithm will converge globally at a linear rate to the unique solution,

\subsection{Failure of Forward differentiation for $C^{1,1}$ objectives}
The Jacobian of $F$ for the Heavy-Ball agorithm (in $x,y$)  is of the form
\begin{align}
        \mathrm{Jac}_F(x,y,\theta) =  \begin{pmatrix}
            (I - \alpha \nabla^2 f_\theta(x) ) + \beta I & - \beta I\\
            I& 0
        \end{pmatrix},
        \label{eq:heavyBallDiff}
\end{align}
when $f$ is $C^2$. If $f$ is $C^{1,1}$, then the Hessian can be replaced by a set-valued conservative Jacobian of the gradient: $J_{\nabla f_\theta}$.

\begin{proof}[of Proposition \ref{prop:hb-fail}]

Recall that the function $f \colon \RR^2 \to \RR$ is given by
\begin{align*}
    f \colon (x,\theta)  \mapsto
    \begin{cases}
        \frac{x^2}{2} & \text{ if } x \geq 0\\
        \frac{x^2}{8} & \text{ if } x < 0.
    \end{cases}
\end{align*}
We have $f'(x) = x$ for $t \geq 0$ and $f'(x) = \frac{x}{4}$ for $t<0$, therefore, $f'$ is $1$-Lipschitz. The Clarke subdifferential of $f'$ is $\{\frac{1}{4}\}$ for $t<0$, $\{1\}$ for $t>0$ and the segment $\left[ \frac{1}{4}, 1\right]$ at $t=0$. Finally, $f$ is $\mu=\frac{1}{4}$ strongly convex and has $L=1$ Lipschitz gradient and the unique fixed point of the Heavy-Ball algorithm applied to $f(\cdot, \theta)$ is $x = y = \theta$.
Choosing $\alpha = 1$, $\beta = 0.75$, we have
\begin{align*}
    \beta < \frac{1}{2}\left( \frac{\mu}{2L} + \sqrt{\frac{\mu^2}{4L^2} + 2} \right) =  \frac{1}{2}\left( \frac{1}{8} + \sqrt{\frac{1}{64} + 2} \right) \simeq 0.77.
\end{align*}
Therefore, the heavy ball algorithm with this choice of parameter converges linearly to the unique solution which is $0$, a fixed point of the iteration mapping.

Set
\begin{align*}
    F(x,y,\theta) = (x - \nabla_x f(x,\theta) + \beta (x - y), x).
\end{align*}
At $(0,0,0)$, the last column of the Jacobian of $F$ is $(0,0)$ and the first two columns are given by
\begin{align*}
    J = \mathrm{conv} \left\{M_1,M_2 \right\},
\end{align*}
where
\begin{align*}
    M_1 &=
    \begin{pmatrix}
        \frac{3}{2} &- \frac{3}{4}   \\
        1&0
    \end{pmatrix}\qquad\qquad
    M_2 =
    \begin{pmatrix}
        \frac{3}{4} &- \frac{3}{4}   \\
        1&0
    \end{pmatrix}.
\end{align*}
Therefore, the Clarke Jacobian of $F$ (with respect to $x,y$) at $(0,0,0)$ is given by
\begin{align*}
J_F(0,0,0) = \conv\{M_1,M_2\}, \qquad  M_1 =
    \begin{pmatrix}
        \frac{3}{2} &- \frac{3}{4}   \\
        1&0
    \end{pmatrix},\qquad
    M_2 =
    \begin{pmatrix}
        \frac{3}{4} &- \frac{3}{4}   \\
        1&0
    \end{pmatrix}.
  \end{align*}
We have
\begin{align*}
    M_1 M_1 M_2 M_2 = \frac{-1}{32}
    \begin{pmatrix}
        36& 0\\
        27&9
    \end{pmatrix},
\end{align*}
which has two eigenvalues $\frac{-9}{8}< -1$ and $\frac{-9}{32}$. Setting for any $\theta \in \RR$ $x_0(\theta) = \theta$, $y_0(\theta) = \theta$, we have for all $k \in \NN$ $x_k(\theta) = y_k(\theta) = \theta$, in other words, this is the unique fixed point of the Heavy-Ball algorithm.

\end{proof}

Given $l \in \NN$, the forward propagation recursion in \eqref{eq:autodiff} presented in Figure \ref{fig:heavyBallInstability} satisfies for $k =8 l$
\begin{align*}
    (M_1 M_1 M_2 M_2)^{2l}
    \begin{pmatrix}
        1\\1
    \end{pmatrix}
\end{align*}
This products will diverge diverge due to the eigenvalue of $(M_1M_1M_2M_2)^2$ strictly above 1. In other words, for all $k$, $J_{x_{8k}}$ given by \eqref{eq:autodiff} contains elements which magnitude diverges at a geometric rate. We conclude that, for all $k \in \NN$, $J_{x_k}$ contains elements which magnitude diverge at a geometric rate.

This illustrates the failure of forward derivative propagation on $f(\cdot,\theta)$: the Heavy Ball algorithm is stable and globally linearly convergent, its fixed point is differentiable (it is actually constant in $\theta$), yet there is a parametric initialization $x(\theta),y(\theta)$ such that forward propagation of derivatives produces diverging elements for $\theta = 0$. Note that implicit differentiation provides the correct derivative, which is $0$, since $x(\theta) = 0$ is the unique fixed point of the gradient iterations. Forward derivative propagation on the gradient descent algorithms also results in the limit in $0$ derivative since it only contains element which converge to $0$ at a geometric rate.

Le us emphasize again that such pathology would not happen if $f$ was $C^2$. Indeed, in this case, $J_f^2$ would be single valued and the divergence phenomenon would not appear. This illustrate a fundamental difference between $C^{1,1}$ and $C^2$ objectives in terms of forward derivative propagation for second order inertial methods.

\section{Experiments details}
\label{app:xp}

All the experiments where run on a MacBook M1 Pro (arm64), on Python 3.9 and \texttt{numpy} 1.21 for a compute time inferior to one hour.
They are repeated 100 times, and we report the median as a blue line and the first and last deciles as a blue shaded area.
The solutions are computed with 2000 iterations, and the curves are reported for the 1000 first iterations.
The differentiation of all methods is performed in forward-mode with \texttt{jacfwd} of the module \texttt{jax}.

\paragraph{Forward--Backward for the Ridge.}
The dimensions of the problem are $n=500$, $p=300$.
The design matrix is Gaussian, i.e., $X_{i,j} \iidsim \mathcal{N}(0,1)$ and the observations $y_i \iidsim \mathcal{N}(0,1)$.
The regularization parameter is set to $\theta = 0.05$.

\paragraph{Forward--Backward algorithm for the Lasso.}
The dimensions of the problem are $n=50$, $p=500$.
The design matrix is Gaussian, i.e., $X_{i,j} \iidsim \mathcal{N}(0,1)$ and the observations $y_i \iidsim \mathcal{N}(0,1)$.
The regularization parameter is set to $\theta = 0.2 \times \theta_{\text{max}}$ where $\theta_{\text{max}} = \| X^\top y \|_{\infty}$.

\paragraph{Douglas--Rachford for the Sparse Inverse Covariance Selection.}
We consider covariance matrices of size $n \times n$ where $n=50$ and $\theta = 0.1$.
The matrix $C$ is generated as $C = V^\top V$ where $V_{i,j} \iidsim \mathcal{N}(0,1)$.

\paragraph{ADMM for Trend Filtering.}
We consider the cyclic 1D Total Variation $n=p=75$ and $\lambda = 3.0$.
Here $\theta \iidsim \mathcal{N}(0,1)$.

\section{Assets used}
\label{app:assets}
Our numerical experiments rely on:
\begin{itemize}
    \item \texttt{numpy}~\cite{harris2020array}, released under BSD-3 license.
    \item \texttt{matplotlib}~\cite{hunter2007matplotlib}, released under PSF license.
    \item \texttt{jax}~\cite{jax2018github}, released under Apache-2.0 license.
\end{itemize}

\end{document}